\documentclass[11pt,letterpaper,reqno]{amsart}
\usepackage{amsmath}
\usepackage{amsfonts}
\usepackage{amssymb}
\usepackage{amsthm}
\usepackage{hyperref}
\usepackage{url}
\usepackage{color}
\usepackage{breakurl}

\newcommand{\fix}[1]{\textbf{\textcolor{red}{#1}}}
\newcommand{\reply}[1]{\textbf{\textcolor{blue}{#1}}}

\usepackage[dvips]{graphics}
\usepackage{epsfig}

\newcommand\be{\begin{equation}}
\newcommand\ee{\end{equation}}
\newcommand\bea{\begin{eqnarray}}
\newcommand\eea{\end{eqnarray}}
\newcommand\bi{\begin{itemize}}
\newcommand\ei{\end{itemize}}
\newcommand\ben{\begin{enumerate}}
\newcommand\een{\end{enumerate}}

\theoremstyle{plain}
\newtheorem{thm}{Theorem}[section]
\newtheorem{lemma}[thm]{Lemma}
\newtheorem{rek}[thm]{Remark}
\newtheorem{cor}[thm]{Corollary}
\newtheorem*{thm*}{Theorem}
\newtheorem{defi}[thm]{Definition}
\newtheorem{lem}[thm]{Lemma}
\newtheorem{conj}[thm]{Conjecture}

\newcommand{\twocase}[5]{#1 \begin{cases} #2 & \text{#3}\\ #4 &\text{#5} \end{cases}   }

\newcommand{\mel}{\mathcal{M}}
\newcommand{\E}{\ensuremath{\mathbb{E}}}
\newcommand{\Q}{\mathbb{Q}}
\newcommand{\N}{\mathbb{N}}
\newcommand{\ga}{\alpha}
\newcommand{\Z}{\mathbb{Z}}
\newcommand{\var}[1]{{\rm Var}\left(#1\right)}

\newcommand{\bburl}[1]{\textcolor{blue}{\url{#1}}}
\newcommand{\blue}[1]{\textcolor{blue}}
\newcommand{\ncr}[2]{{#1 \choose #2}}

\numberwithin{equation}{section}

\begin{document}

\title{Benford's Law and Continuous Dependent Random Variables}

\author[Becker et. al.]{Thealexa Becker}\email{\textcolor{blue}{\href{mailto:tbecker@smith.edu}{tbecker@smith.edu}}}
\address{Smith College, Northampton, MA 01063}

\author[]{David Burt}\email{\textcolor{blue}{\href{mailto:drb3@williams.edu}{drb3@williams.edu}}} \address{Williams College, Williamstown, MA 01267}

\author[]{Taylor C. Corcoran}\email{\textcolor{blue}{\href{mailto:taylorc3@email.arizona.edu}{taylorc3@email.arizona.edu}}}
\address{The University of Arizona, Tucson, AZ 85721}

\author[]{Alec Greaves-Tunnell}\email{\textcolor{blue}{\href{mailto:ahg1@williams.edu}{ahg1@williams.edu}}} \address{Williams College, Williamstown, MA 01267}

\author[]{Joseph R. Iafrate}\email{\textcolor{blue}{\href{mailto:Joseph.R.Iafrate@williams.edu}{Joseph.R.Iafrate@williams.edu}}}
\address{Williams College, Williamstown, MA 01267}

\author[]{Joy Jing}\email{\textcolor{blue}{\href{mailto:Joy.Jing@williams.edu}{Joy.Jing@williams.edu}}}
\address{Williams College, Williamstown, MA 01267}

\author[]{Steven J. Miller}\email{\textcolor{blue}{\href{mailto:sjm1@williams.edu}{sjm1@williams.edu}},  \textcolor{blue}{\href{Steven.Miller.MC.96@aya.yale.edu}{Steven.Miller.MC.96@aya.yale.edu}}}
\address{Department of Mathematics and Statistics, Williams College, Williamstown, MA 01267}

\author[]{Jaclyn D. Porfilio}\email{\textcolor{blue}{\href{mailto:Jaclyn.D.Porfilio@williams.edu}{Jaclyn.D.Porfilio@williams.edu}}}
\address{Williams College, Williamstown, MA 01267}

\author[]{Ryan Ronan}\email{\textcolor{blue}{\href{mailto:ryan.p.ronan@gmail.com}{ryan.p.ronan@gmail.com}}}
\address{Cooper Union, New York, NY 10003}

\author[]{Jirapat Samranvedhya}\email{\textcolor{blue}{\href{mailto:Jirapat.Samranvedhya@williams.edu}{Jirapat.Samranvedhya@williams.edu}}}
\address{Williams College, Williamstown, MA 01267}

\author[]{Frederick W. Strauch}\email{\textcolor{blue}{\href{mailto:fws1@williams.edu}{fws1@williams.edu}}}
\address{Department of Physics, Williams College, Williamstown, MA 01267}

\author[]{Blaine Talbut}\email{\textcolor{blue}{\href{mailto:blainetalbut@math.ucla.edu}{blainetalbut@math.ucla.edu}}} \address{University of California Los Angeles, Los Angeles, CA 90095}

\subjclass[2010]{60A10, 11K06  (primary),  (secondary) 60E10}

\keywords{Benford's Law, Fourier transform, Mellin transform, dependent random variables, fragmentation}

\date{\today}

\thanks{This work was supported by NSF grants DMS0850577, DMS0970067, DMS1265673, DMS1561945 and PHY1005571, a Research Corporation Cottrell College Science Award, and Williams College. We thank Don Lemons, Charles Wohl, and participants of numerous Williams SMALL REUs for helpful discussions. This paper is dedicated to the memory of Corey Manack, a wonderful colleague who was generous with his time and provided many helpful comments in the course of these investigations.}

\begin{abstract}
Many mathematical, man-made and natural systems exhibit a leading-digit bias, where a first digit (base 10) of 1 occurs not 11\% of the time, as one would expect if all digits were equally likely, but rather 30\%. This phenomenon is known as Benford's Law. Analyzing which datasets adhere to Benford's Law and how quickly Benford behavior sets in are the two most important problems in the field. Most previous work studied systems of independent random variables, and relied on the independence in their analyses.

Inspired by natural processes such as particle decay, we study the dependent random variables that emerge from models of decomposition of conserved quantities. We prove that in many instances the distribution of lengths of the resulting pieces converges to Benford behavior as the number of divisions grow, and give several conjectures for other fragmentation processes. The main difficulty is that the resulting random variables are dependent, which we handle by using tools from Fourier analysis and irrationality exponents to obtain quantified convergence rates. Our method can be applied to many other systems; as an example, we show that the $n!$ entries in the determinant expansions of $n\times n$ matrices with entries independently drawn from nice random variables converges to Benford's Law.
\end{abstract}

\maketitle

\tableofcontents

\section{Introduction}

\subsection{Background}\label{sec:process}

In 1881, American astronomer Simon Newcomb \cite{New} noticed that the earlier pages of logarithm tables, those which corresponded to numbers with leading digit 1, were more worn than other pages. He proposed that the leading digits of certain systems are logarithmically, rather than uniformly, distributed. In 1938, Newcomb's leading digit phenomenon was popularized by physicist Frank Benford, who examined the distribution of leading digits in datasets ranging from street addresses to molecular weights. The digit bias investigated by these scientists is now known as Benford's Law.

Formally, a dataset is said to follow  Benford's Law base $B$ if the probability of observing a leading digit $d$ base $B$ is $\log_{B} (\frac{d+1}{d})$; thus we would have a leading digit of $1$ base 10 approximately $30\%$ of the time, and a leading digit of $9$ less than $5\%$ of the time. More generally, we can consider all the digits of a number. Specifically, given any $x > 0$ we can write it as \be x \ = \ S_B(x) \cdot 10^{k_B(x)}, \ee where $S_B(x) \in [1, B)$ is the significand of $x$ and $k_B(x)$ is an integer; note two numbers have the same leading digits if their significands agree. Benford's Law is now the statement that ${\rm Prob}(S_B(x) \le s) = \log_B(s)$.

Benford's Law arises in applied mathematics \cite{BH1}, auditing \cite{DrNi,MN3,Nig1,Nig2,Nig3,NMi}, biology \cite{CLTF}, computer science \cite{Knu}, dynamical systems \cite{Ber1,Ber2,BBH,BHKR,Rod}, economics \cite{To}, geology \cite{NM}, number theory \cite{ARS,KonMi,LS}, physics \cite{PTTV}, signal processing \cite{PHA}, statistics \cite{MN2,CLM} and voting fraud detection \cite{Meb}, to name just a few. See \cite{BH2,Hu} for extensive bibliographies and \cite{BH3,BH4,BH5, BH6, Dia,Hi1,Hi2,JKKKM,JR,MN1,Pin,Rai} for general surveys and explanations of the Law's prevalence, as well as the book edited by Miller \cite{Mil}, which develops much of the theory and discusses at great length applications in many fields.

One of the most important questions in the subject, as well as one of the hardest, is to determine which processes lead to Benford behavior. Many researchers \cite{Adh,AS,Bh,JKKKM,Lev1,Lev2,MN1,Rob,Sa,Sc1,Sc2,Sc3,ST} observed that sums, products and in general arithmetic operations on random variables often lead to a new random variable whose behavior is closer to satisfying Benford's law than the inputs, though this is not always true (see \cite{BH5}). Many of the proofs use techniques from measure theory and Fourier analysis, though in some special cases it is possible to obtain closed form expressions for the densities, which can be analyzed directly. In certain circumstances these results can be interpreted through the lens of a central limit theorem law; as we only care about the logarithms modulo 1, the Benfordness follows from convergence of this associated density to the uniform distribution (see for example \cite{MN1} or Chapter 3 of \cite{Mil}).

A crucial input in many of the above papers is that the random variables are \emph{independent}. In this paper we explore situations where there are dependencies. The dependencies we investigate are different than many others in the literature. For example, previous work studied dynamical systems and iterates or powers of a given random variable, where once the initial seed is chosen the resulting process is deterministic; see for example \cite{AS,BBH,Dia,KonMi,LS}. In our systems instead of having just one random variable we have a large number of independent random variables $N$ generating an enormous number of dependent random variables $M$ (often $M = 2^N$, though in one of our examples involving matrices we have $M = N!$).

Our introduction to the subject of this paper came from reading an article of Lemons \cite{Lem} (though see the next subsection for other related problems), who studied the decomposition of a conserved quantity; for example, what happens during certain types of particle decay. As the sum of the piece sizes must equal the original length, the resulting summands are clearly dependent. While it is straightforward to show whether or not individual pieces are Benford, the difficulty is in handling all the pieces simultaneously. We comment in greater detail about Lemons' work in Appendix \ref{sec:noteslemons}.

In the next subsection we describe some of the systems we study. In analyzing these problems we develop a technique to handle certain dependencies among random variables, which we then show is applicable in other systems as well.

\subsection{1-Dimensional Decomposition Models and Notation}\label{sec:problemsandnotation}

The techniques we develop to show Benford behavior for problems with dependent random variables are applicable to many systems. In the interest of space, we will describe in detail here just three variations of a stick decomposition, and later discuss conjectures about other possible decomposition processes and some results in higher dimensions. As an example of the power of this approach we also prove that the leading digits of the $n!$ terms in the determinant expansion of a matrix whose entries are independent, identically distributed `nice' random variables follow a Benford distribution as $n$ tends to infinity.

There is an extensive literature on decomposition problems; we briefly comment on some other systems that have been successfully analyzed and place our work in context. Kakutani \cite{Ka} considered the following deterministic process. Let $Q_0 = \{0, 1\}$ and given $Q_k = \{x_0 = 0, x_1, \dots, x_k = 1\}$ (where the $x_i$'s are in increasing order) and an $\alpha \in (0,1)$, construct $Q_{k+1}$ by adding points $x_i + \alpha (x_{i+1}-x_i)$ in each subinterval $[x_i, x_{i+1}]$ where $x_{i+1}-x_i = \max_{1 \le \ell \le k-1} |x_{\ell+1}-x_\ell|$. He proved that as $k\to\infty$, the points of $Q_k$ become uniformly distributed, which implies that this process is non-Benford. This process has been generalized; see for example \cite{AF,Ca,Lo,PvZ,Sl,vZ} and the references therein, and especially the book \cite{Bert}. See also \cite{Kol} for processes related to particle decomposition, \cite{CaVo,Ol} for 2-dimensional examples, and \cite{IV} for a fractal setting.

Most of this paper is devoted to the following decomposition process, whose first few levels are shown in Figure \ref{fig:tree3levels}. Begin with a stick of length $L$, and a density function $f$ on $(0,1)$; all cuts will be drawn from this density. Cut the stick at proportion $p_1$. This is the first level, and results in two sticks. We now cut the left fragment at proportion $p_2$ and the right at proportion $p_3$. This process continues for $N$ iterations. Thus if we start with one stick of length $L$, after one iteration we have sticks of length $Lp_1$ and $L(1-p_1)$, after two iterations we have sticks of length $Lp_1p_2$, $Lp_1(1-p_2)$, $L(1-p_1)p_3$, and $L(1-p_1)(1-p_3)$, and so on. Iterating this process $N$ times, we are left with $2^N$ sticks.

\begin{figure}[h]
\begin{center}
\scalebox{.7}{\includegraphics{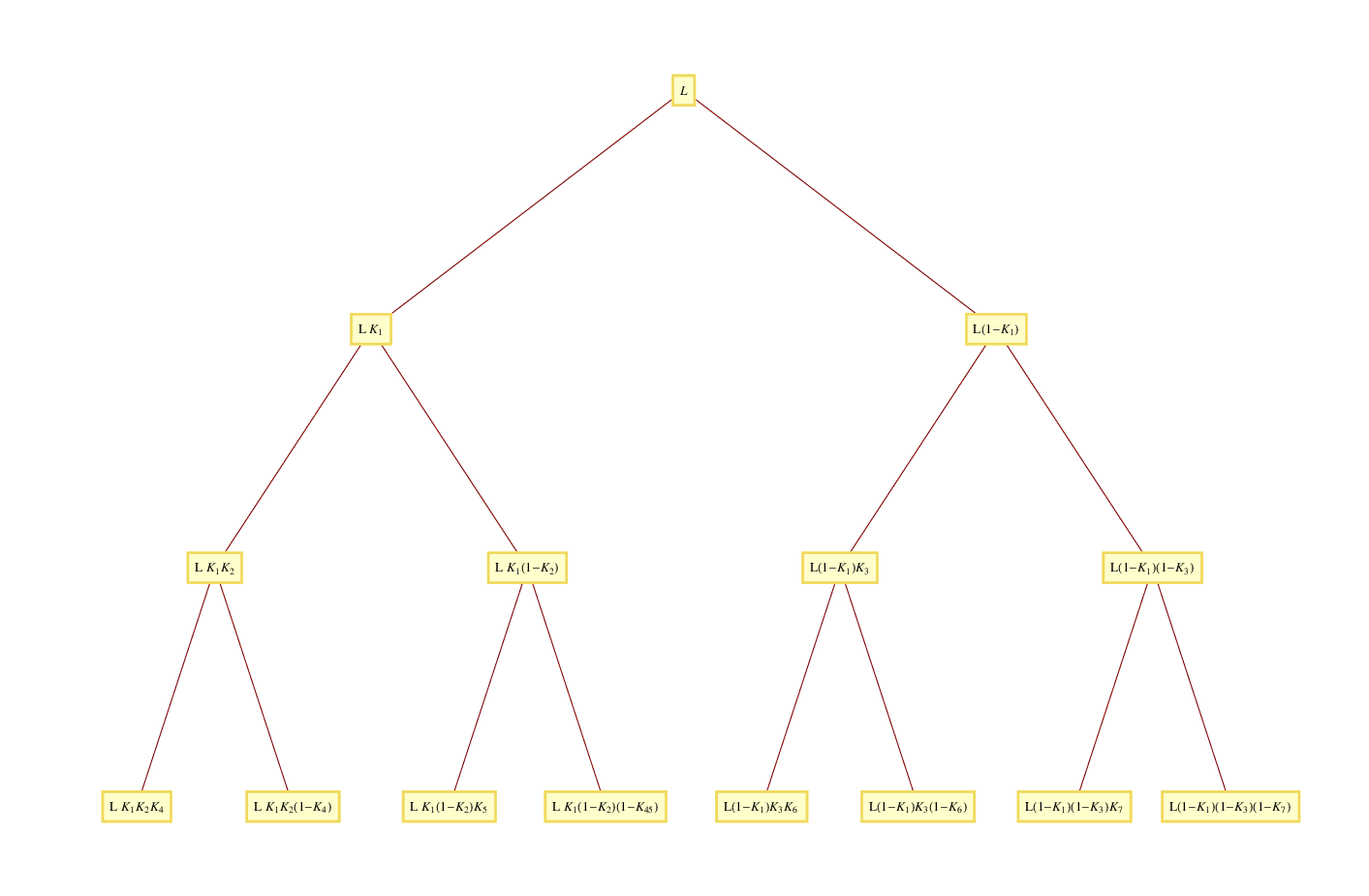}}
\caption{\label{fig:tree3levels} Unrestricted Decomposition: Breaking $L$ into pieces, $N=3$. }
\end{center}\end{figure}

We analyze whether the lengths of the resulting pieces follow Benford's Law for different choices of $f$, as well as modifications of the fragmentation procedure. This process builds on earlier work in the field, which we discuss after describing our systems.

\begin{enumerate}

\item \texttt{Unrestricted Decomposition Model}: As described above, each proportion is drawn from a distribution $f$ and all pieces decompose.

\item \texttt{Restricted Decomposition Model}: Proportions are chosen as in Case (1), but only one of the two resulting pieces from each iteration decomposes further.

\item \texttt{Fixed Proportion Decomposition}: All pieces decompose, but a fixed proportion $p$ is chosen prior to the decomposition process and is used for all sticks during every iteration.

\end{enumerate}

In addition to similarities with the work mentioned above, the last problem is a similar to a fragmentation tree model investigated by Janson and Neininger \cite{JN}. Phrasing their work in our language, they randomly chose and fixed $b$ probabilities $p_1, \dots, p_b$ and then at each stage each piece of length $x$ split into $b$ pieces of length $p_1 x, \dots, p_b x$, unless $x$ is below a critical threshold in which case the piece is never decomposed further. They were interested in the number of pieces after their deterministic process ended, whereas we are interested in the distribution of the leading digits of the lengths. While it is possible to apply some of their results to attack our third model, the problem can be attacked directly. The situation here is similar to other problems in the field. For example, Miller and Nigrini \cite{MN1} prove that certain products of random variables become Benford. While it is possible to prove this by invoking a central limit theorem type argument, it is not necessary as we do not need to know the distribution of the product completely, but rather we only need to know the distribution of its logarithm modulo 1. Further, by not using the central limit theorem they are able to handle more general distributions; in particular, they can handle random variables with infinite variance.

Before we can state our results, we first introduce some notation which is needed to determine which $f$ lead to Benford behavior.

\begin{defi}[Mellin Transform, $\mathcal{M}_f(s)$]
Let $f(x)$ be a continuous real-valued function on $[0,\infty)$.\footnote{As our random variables are proportions, for us $f$ is always a probability density with support on the unit interval.} We define its Mellin transform\footnote{Note $\mathcal{M}_f(s) = \E[x^{s-1}]$, and thus results about expected values translate to results on Mellin transforms; as $f$ is a density $\mathcal{M}_f(1) = 1$. Letting $x = e^{2\pi u}$ and $s=\sigma - i\xi$ gives $\mel_f(\sigma - i\xi) = 2\pi \int_{-\infty}^\infty \left( f(e^{2\pi u}) e^{2\pi \sigma u}\right) e^{-2\pi i u \xi} du$, which is the Fourier transform of $g(u) = 2\pi f(e^{2\pi u})e^{2\pi\sigma u}$. The Mellin and Fourier transforms as thus related; this logarithmic change of variables explains why both enter into Benford's Law problems. We can therefore obtain proofs of Mellin transform properties by mimicking the proofs of the corresponding statements for the Fourier transform; see \cite{SS1,SS2}.}, $\mathcal{M}_f(s)$, by \be \mathcal{M}_f(s) \ := \ \int_0^\infty f(x) x^s \frac{dx}{x}. \ee
\end{defi}

We next define the significand indicator function; while we work base 10, analogous definitions hold for other bases.

\begin{defi}[Significand indicator function, $\varphi_s$] For $s \in [1,10)$, let \be\label{def:phissig} \twocase{\varphi_s(x) \ := \ }{1}{{\rm if\ the\ significand\ of\ $x$\ is\ at\ most\ $s$}}{0}{{\rm otherwise;}}
\ee
thus $\varphi_s$ is the indicator function of the event of a significand at most $s$.
\end{defi}

In all proofs, we label the set of stick lengths resulting from the decomposition process by $\{X_i\}$. Note that a given stick length can occur multiple times, so each element of the set $\{X_i\}$ has associated to it a frequency.

\begin{defi}[Stick length proportions, $P_N$] Given stick lengths $\{X_i\}$, the proportion whose significand is at most $s$, $P_N(s)$, is
\bea
P_N(s)\ :=\ \frac{\sum\limits_i \varphi_s(X_i)}{\#\{X_i\}}.
\eea
\end{defi}

In the \texttt{Fixed Proportion Decomposition Model}, we are able to quantify the rate of convergence if $\log_{10} \frac{1-p}{p}$ has finite irrationality exponent.

\begin{defi}[Irrationality exponent] A number $\ga$ has irrationality exponent $\kappa$ if $\kappa$
is the supremum of all $\gamma$ with \be\label{eq:irrtypedef}\underline{\lim}_{q\to\infty} q^{\gamma +
1} \min_p \left| \ga - \frac{p}{q}\right|\ =\ 0.\ee \end{defi}

By Roth's theorem, every algebraic irrational has irrationality exponent $1$. See for example \cite{HS, MT-B, Ro} for more details.\\ \

Finally, we occasionally use big-Oh and little-oh notation. We write $f(x) = O(g(x))$ (or equivalently $f(x) \ll g(x)$) if there exists an $x_0$ and a $C>0$ such that, for all $x \ge x_0$, $|f(x)| \le Cg(x)$, while $f(x) = o(g(x))$ means $\lim_{x\to\infty} f(x)/g(x) = 0$.

\subsection{Results}\label{sec:results}

We state our results for the fragmentation models of \S\ref{sec:problemsandnotation} and some generalizations. While a common way of proving a sequence is Benford base $B$ is to show that its logarithms base $B$ are equidistributed\footnote{A sequence $\{x_n\}$ is equidistributed modulo 1 if for any $(a,b) \subset (0,1)$ we have $\lim_{N\to \infty} \frac1N \cdot \#\{n \le N: x_n \in (a,b)\} = b-a$.} (see, for example, \cite{Dia,MT-B}), as we are using the Mellin transform and not the Fourier transform we instead often directly analyze the significand function. To show that the first digits of $\{X_{i}\}$ follow a Benford distribution, it suffices to show that
\begin{enumerate}
\item $\lim\limits_{N \to \infty} \E[P_n(s)] \ = \ \log_{10}(s)$, and
\item $ \lim\limits_{N \to \infty} \var{P_n(s)} \ = \ 0$.
\end{enumerate}

Viewing $P_n(s)$ as the cumulative distribution function of the process, the above shows that we have convergence in distribution\footnote{A sequence of random variables $R_1, R_2, \dots$ with corresponding cumulative distribution functions $F_1, F_2, \dots$ \emph{converges in distribution} to a random variable $R$ with cumulative distribution $F$ if $\lim_{n\to\infty} F_n(r) = F(r)$ for each $r$ where $F$ is continuous.} to the Benford cumulative distribution function (see \cite{GS}).


For ease of exposition and proof we often concentrate on the uniform distribution case, and remark on generalizations. In our proofs the key technical condition is that the densities satisfy \eqref{eq:summeltransformsnozero} below. This is a very weak condition if the densities are fixed, essentially making sure we stay away from random variables where the logarithm modulo 1 of the densities are supported on translates of subgroups of the unit interval. If we allow the densities to vary at each stage, it is still a very weak condition but it is possible to construct a sequence of densities so that, while each one satisfies the condition, the sequence does not. We give an example in Appendix \ref{benfordnotbenford}; see \cite{MN1} for more details.

\subsubsection{1-Dimensional Results} \ \\

\begin{thm}[Unrestricted 1-Dimension Decomposition Model]\label{thm:unrestricted} Fix a continuous probability density $f$ on $(0,1)$ such that
\be\label{eq:summeltransformsnozero} \lim_{N\to\infty} \sum_{\ell = -\infty \atop \ell \neq 0}^\infty \prod_{m=1}^N \mathcal{M}_h\left(1-\frac{2\pi i \ell}{\log 10}\right) \ = \ 0,
\ee
where $h(x)$ is either $f(x)$ or $f(1-x)$ (the density of $1-p$ if $p$ has density $f$).  Given a stick of length $L$, independently choose cut proportions $p_1, p_2, \dots, p_{2^N-1}$ from the unit interval according to the probability density $f$. After $N$ iterations we have
\bea X_1 &\ =\ & Lp_1p_2p_4\dotsm p_{2^{N-2}}p_{2^{N-1}} \nonumber\\
X_2 & \ = \ & Lp_1p_2p_4 \dotsm p_{2^{N-2}} (1-p_{2^{N-1}}) \nonumber\\
&\vdots & \nonumber\\
X_{2^N-1} &\ = \ &  L(1-p_1)(1-p_3)(1-p_7) \dotsm (1-p_{2^{N-1}-1}) p_{2^N-1} \nonumber\\
X_{2^N} &\ = \ &   L(1-p_1)(1-p_3)(1-p_7) \dotsm (1-p_{2^{N-1}-1}) (1-p_{2^N-1}),
\eea
and \be P_N(s) \ := \  \frac{ \sum_{i = 1}^{2^N} \varphi_s(X_i)}{2^N}\ee is the fraction of partition pieces $X_1, \dots, X_{2^N}$ whose significand is less than or equal to $s$ (see \eqref{def:phissig} for the definition of $\varphi_s$).  Then

\begin{enumerate}

\label{expectation} \item $\displaystyle \lim _{N \to \infty} \mathbb{E} [P_N(s)] = \log_{10}s$,

\label{variance} \item $\displaystyle \lim _{N \to \infty} {\rm Var}\left(P_N(s)\right) = 0$.

\end{enumerate}

Thus as $N \to \infty$, the significands of the resulting stick lengths converge in distribution to Benford's Law.
\end{thm}

\begin{rek} Theorem \ref{thm:unrestricted} can be greatly generalized. We assumed for simplicity that at each stage each piece must split into exactly two pieces. A simple modification of the proof shows Benford behavior is also attained in the limit if at each stage each piece independently splits into $0, 1, 2, \dots$ or $k$ pieces with probabilities $q_0, q_1, q_2, \dots, q_k$ (so long as $q_0 < 1$). Furthermore, we do not need to use the same density for each cut, but can draw from a finite set of densities that satisfy the Mellin transform condition. Interestingly, we can construct a counter-example if we are allowed to take infinitely many distinct densities satisfying the Mellin condition; we give one in Appendix \ref{benfordnotbenford}.\footnote{The reason we need to be careful is that, while typically products of independent random variables converge to Benford behavior, there are pathological choices where this fails (see Example 2.4 of \cite{MN1}).}
\end{rek}

The discrete analogue of the unrestricted model also results in Benford behavior asymptotically.
\begin{thm}\label{thm:discrete}
Consider the following fragmentation process: Start with a rod of integer length $\ell=\ell_0$, in iteration $k$ select an integer $X_k \in [1,\ell_k]$ with uniform probability and fracture the rod so that you are left with a piece of length $X_k$ and a piece of length $\ell_i-X_i=:\ell_{i+1}$. Continue this process until $\ell_{k}=0$. Then if $\{X_\ell\}$ denotes the (random) set of the $X_k$, as $\ell\to \infty$, the distribution of $\{X_\ell\}$ converges to (strong) Benford behavior with probability $1$.
\end{thm}

\begin{rek}
The techniques used in the proof of Theorem \ref{thm:discrete} generalize naturally to a wide class of integer-valued probability functions. In particular similar proofs should work for density functions that produce a large number of fragments with high probability and can be well approximated by a continuous process satisfying the Mellin Transform property of the previous section.
\end{rek}

\begin{thm}[Restricted 1-Dimensional Decomposition Model] \label{restricted thm}
Start with a stick of length $L$, and cut this stick at a proportion $p_1$ chosen uniformly at random from $(0, 1)$. This results in two sticks, one length $Lp_1$ and one of length $L(1-p_1)$. Do not decompose the stick of length $L(1-p_1)$ further, but cut the other stick at proportion $p_2$ also chosen uniformly from the unit interval. The resulting sticks will be of lengths $p_1p_2$ and $p_1(1-p_2$). Again do not decompose the latter stick any further. Recursively repeat this process N-1 times,leaving $N$ sticks:
\bea X_1 &\ =\ & L(1-p_1) \nonumber\\
X_2 & \ = \ & Lp_1(1-p_2) \nonumber\\
&\vdots & \nonumber\\
X_{N-1} &\ = \ &   Lp_1p_2 \dotsm p_{N-2}(1-p_{N-1}).\nonumber \\
X_N &\ = \ &  Lp_1p_2 \dotsm p_{N-1}.
\eea
The distribution of the leading digits of these resulting $N$ sticks converges in distribution to Benford's Law.
\end{thm}

\begin{rek} We may replace the uniform distribution with any nice distribution that satisfies the Mellin transform condition of \eqref{eq:summeltransformsnozero}. \end{rek}

\begin{thm}[Fixed Proportion 1-Dimensional Decomposition Model] \label{fixed proportion thm}
Choose any $p \in (0,1)$. In Stage 1, cut a given stick at proportion $p$ to create two pieces. In Stage 2, cut each resulting piece into two pieces at the same proportion $p$. Continue this process $N$ times, generating $2^N$ sticks with $N+1$ distinct lengths (assuming $p \neq 1 /2 $) given by
\bea x_1 &\ =\ & Lp^N \nonumber\\
x_2 & \ = \ & Lp^{N-1}(1-p)\nonumber\\
x_3 & \ = \ & Lp^{N-2}(1-p)^2\nonumber\\
&\vdots & \nonumber\\
x_{N} &\ = \ &  Lp(1-p)^{N-1}\nonumber\\
x_{N+1} &\ = \ &   L(1-p)^N,
\eea
where the frequency of $x_n$ is ${N \choose n}/2^N$. Choose $y$ so that $10^y = (1-p)/p$, which is the ratio of adjacent lengths (i.e., $x_{i+1}/x_i$). The decomposition process results in stick lengths that converge in distribution to Benford's Law if and only if $y\not\in\Q$. If $y$ has finite irrationality exponent, the convergence rate can be quantified in terms of the exponent, and there is a power savings.
\end{thm}

\subsubsection{2-Dimensional Results} \ \\

The next two results are indicative of what can be done in higher dimensions.

\begin{thm}[Unrestricted Decomposition of Triangles]\label{thm:triangle}
Consider the following two-dimensional decomposition process. Beginning with a triangle of area $A$, select a point in the interior according to some probability distribution and connect this point to each of the three vertices to obtain three sub-triangles. Now independently select a point in the interior of each of these triangles and repeat this process until there are $3^N$ triangles.
Fix a continuous probability density $f$ on a triangular region, $T_0$. Define $\tilde{f} := (f_1,f_2,1-f_1-f_2)$. Let $\mathcal{X}_N$ denote the set of areas of the sub-triangles that result from $N$ iterations of the decomposition of $T_0$ described above. Then, if
\be\label{eq:condfortriangles} \lim_{N \to \infty} \sum_{\substack{\ell = -\infty \\ \ell \neq 0}}^\infty \prod_{m=1}^N
	\mathcal{M} \tilde{f} \left( 1 - \frac{2 \pi i \ell}{\log 10} \right), \ee
then the significands of the areas in $\mathcal{X}$ converge in distribution to Benford's law. \fix{BLAINE: NEED TO FIX HERE. WHAT DO YOU WANT THIS LIMIT TO BE? YOU HAVE A PRODUCT OVER $m$ BUT NO $m$ DEPENDENCE; ARE YOU CHOOSING FUNCTIONS FROM THE TRIPLE? IS THE TRIPLE A VECTOR? NEED TO BE MORE CAREFUL. IN OTHER PROOFS BELOW WE ALSO HAVE A PRODUCT OVER $m$ BUT NO $m$ DEPENDENCE, SO THAT'S FINE, BUT JUST NEED TO BE CLEAR.}
\fix{I THINK THIS IS WHAT IS MEANT IN THE MELLIN CONDITION, BUT NOT CLEAR IT IS WELL DEFINED, HOW IT RELATES TO THE OTHER DEFINITIONS \be \tilde{f}(x)\ =\ \int_{y_1 = - \infty}^\infty \cdots \int_{y_{k-1} = - \infty}^\infty f(x,y_1,\ldots,y_{k-1}) d y_1 \cdots d y_{k-1}.\ee}
\end{thm}

\begin{rek}
Note that the condition \eqref{eq:condfortriangles} is again quite weak, and holds if $\tilde{f}$ has finite moments. \fix{DAVID: FINITE OR BOUNDED?} \reply{BLAINE: I THINK BOUNDED -DavidB}
\end{rek}

Affine mappings do not exist between arbitrary quadrilaterals, as affine mappings preserve parallel lines. However, there are continuous mappings from arbitrary convex quadrilaterals to squares that preserve the ratio of areas; the construction of Gromov after Knothe provides such a mapping with other nice properties \cite{Gr}.

\begin{thm}[Unrestricted Decomposition of Quadrilaterals]\label{thm:quadrilateral}
Start with the unit square, a continuous probability density $f$ on $(0,1)$, and a continuous probability density $g$ on $(0,1)^2$. We independently select a point on each side according to $f$. Call these $A,B,C,D$. We then choose a point $E$ in the interior of the quadrilateral $ABCD$ according to the composition of $g$ with a mapping that preserves ratios of areas. We now connect $E$ to each of $A,B,C,D$ in order to decompose the square into four convex quadrilaterals. We then perform the same decomposition independently on each of these quadrilaterals, repeating this process until we obtain $4^N$ quadrilaterals. \fix{BLAINE: NEED TO BE A BIT CAREFUL AS TO HOW WE CHOOSE THE NEW POINTS IN THE NEW QUADRILATERALS, AS THESE WON'T BE SQUARES. THIS IS WHY CAN'T JUST USE THE $g$ ON $(0,1)^2$. NEED A BIT MORE DETAIL HERE.} \fix{I THINK THEY SHOULD BE CHOSEN ACCORDING TO THE COMPOSITION OF G WITH A FIXED MAPPING (DEPENDING ON THE QUADRILATERAL)}Suppose $f$ and $g$ have finite \fix{BDD?} moments. Let $\mathcal{X}_N$ denote the set of areas of the sub-quadrilaterals that result from $N$ iterations of the decomposition of the unit square described above. Then, as $N \to \infty$, the significands of the areas in $\mathcal{X}_N$ converge in distribution to Benford's law.
\end{thm}

\begin{rek}
A more general version of Theorem \ref{thm:quadrilateral} is true, with $f$ and $g$ satisfying a complicated Mellin condition that is implied by finite moments.
\end{rek}

\subsection{Sketch of Proofs} \ \\

We briefly comment on the proofs. We proceed by quantifying the dependencies between the various fragments, and showing that the number of pairs that are highly dependent is small. This technique is applicable to a variety of other systems, and we give another example below. These dependencies introduce complications which prevent us from proving our claims by directly invoking standard theorems on the Benfordness of products. For example, we cannot use the well-known fact that powers of an irrational number $r$ are Benford to prove Theorem \ref{fixed proportion thm} because we must also take into account how many pieces we have of each fragment (equivalently, how many times we have $r^m$ as a function of $m$). We provide arguments in greater detail than is needed for the proofs so that, if someone wished to isolate out rates of convergence, that could be done with little additional effort. While optimizing the errors is straightforward, doing so clutters the proof and can have computations very specific to the system studied, and thus we have chosen not to extract the best possible error bounds in order to keep the exposition as simple as possible.

We end with the promised example of another system where our techniques are applicable. The proof, given in \S\ref{sec:determinant}, utilizes the same techniques as that of the stick decomposition. We again have a system with a large number of independent random variables, $n$, leading to an enormous number of dependent random variables, $n!$.

\begin{thm}\label{thm:detpiecesbenford}
Let $A$ be an $n \times n$ matrix with independent, identically distributed  entries $a_{ij}$ drawn from a distribution $X$ with density $f$. The distribution of the significands of the $n!$ terms in the determinant expansion of $A$ converge in distribution to  Benford's Law if \eqref{eq:summeltransformsnozero} holds with $h=f$.
\end{thm}

\section{Proof of Theorem \ref{thm:unrestricted}: Unrestricted Decomposition} \label{sec:unrestricted}

A crucial input in this proof is a quantified convergence of products of independent random variables to Benford behavior, with the error term depending on the Mellin transform. We use Theorem 1.1 of \cite{JKKKM} (and its generalization, given in Remark 2.3 there); for the convenience of the reader we quickly review this result and its proof in Appendix A of \cite{B--} (the expanded arXiv version of this paper). The dependencies of the pieces is a major obstruction; we surmount this by breaking the pairs into groups depending on how dependent they are (specifically, how many cut proportions they share).

\begin{rek} The key condition in Theorem \ref{thm:unrestricted}, \eqref{eq:summeltransformsnozero}, is extremely weak and is met by most distributions. For example, if $f$ is the uniform density on $(0,1)$ then
\be \mel_f \left(1-\frac{2\pi i \ell}{\log 10}\right) \ = \ \left(1 - \frac{2\pi i \ell}{\log 10}\right)^{-1},
\ee
which implies
\be \lim_{N\to\infty} \left| \sum_{\ell = -\infty \atop \ell \neq 0}^\infty \prod_{m=1}^N \mathcal{M}_f\left(1-\frac{2\pi i \ell}{\log 10}\right) \right| \ \le \ 2 \lim_{N\to\infty} \sum_{\ell=1}^\infty \left|1 - \frac{2\pi i \ell}{\log 10}\right|^{-N} \ = \ 0 \ee (we wrote the condition as $\prod_{m=1}^N \mel_f$ instead of $\mel_f^N$ to highlight where the changes would surface if we allowed different densities for different cuts). While this condition is weak, it is absolutely necessary to ensure convergence to Benford behavior; see Appendix \ref{benfordnotbenford}.
\end{rek}

To prove convergence in distribution to Benford's Law, we first prove in \S\ref{unrestricted expval} that $\mathbb{E}[P_N(s)] = \log_{10} s$, and then in \S\ref{sec:varianceunrestricted} prove that ${\rm Var}\left(P_N(s)\right) \to 0$; as remarked earlier these two results yield the desired convergence. The proof of the mean is significantly easier than the proof of the variance as expectation is linear, and thus there are no issues from the dependencies in the first calculation, but there are in the second. The key contribution of this work is quantifying how often certain dependencies can arise, which leads to a tractable analysis.

\subsection{Expected Value}\label{unrestricted expval}

\begin{proof}[Proof of Theorem \ref{thm:unrestricted} (Expected Value)] By linearity of expectation,
\begin{equation}
\mathbb{E}[P_N(s)] \ = \  \mathbb{E}\left[ \frac{ \sum _{i = 1} ^{2^N} \varphi_s(X_i)}{2^N} \right] \ = \  \frac{1}{2^N} \sum_{i = 1} ^{2^N} \mathbb{E} [\varphi_s(X_i)]. \label{expect}
\end{equation}

We recall that all pieces can be expressed as the product of the starting length $L$ and cutting proportions $p_i$. While there are dependencies among the lengths $X_i$, there are no dependencies among the $p_i$'s. A given stick length $X_i$ is determined by some number of factors $k$ of $p_i$ and $N-k$ factors of $1-p_i$ (where $p_i$ is a cutting proportion between $0$ and $1$ drawn from a distribution with density $f$). By relabeling if necessary, we may assume \be X_i\ =\ L p_1 p_2 \dotsm p_{k} (1-p_{k+1}) \dotsm (1-p_N); \ee the first $k$ proportions  are drawn from a distribution with density $f(x)$ and the last $N-k$ from a distribution with density $f(1-x)$.

The proof is completed by showing $\lim_{N\to\infty}\E[\varphi_s(X_i)] = \log_{10} s$. We have
\begin{eqnarray}
\mathbb{E}[\varphi_s(X_i)] & \ = \ & \int_{p_1 = 0}^{1} \int_{p_2 = 0}^{1} \dotsm \int_{p_N = 0}^{1} \varphi_s \left( L \prod_{r = 1}^{k} p_r \prod_{m=k+1}^N (1-p_m) \right)\nonumber\\ & & \ \ \ \ \cdot \ \prod_{r=1}^{k} f(p_r) \prod_{m=k+1}^N f(1-p_m) \ dp_1 dp_2 \dotsm dp_N.
\end{eqnarray}

This is equivalent to studying the distribution of a product of $N$ independent random variables (chosen from one of two densities) and then rescaling the result by $L$. The convergence of $L \prod_{r = 1}^{k} p_r \prod_{m=k+1}^N (1-p_m)$ $=$ $X_i$ to Benford follows from \cite{JKKKM} (the key theorem is summarized for the reader's convenience in Appendix A in \cite{B--}, the expanded arXiv version of this paper). We find $\E[\varphi_s(X_i)]$ equals $\log_{10} s$ plus a rapidly decaying $N$-dependent error term. This is because the Mellin transforms (with $\ell \neq 0$) are always less than 1 in absolute value. Thus the error is bounded by the maximum of the error from a product with $N/2$ terms with density $f(x)$ or a product with $N/2$ terms with density $f(1-x)$ (where the existence of $N/2$ such terms follows from the pigeonhole principle). Thus $\lim _{N \to \infty} \mathbb{E}[P_N(s)] = \log_{10} s$, completing the proof.
\end{proof}

\begin{rek} For specific choices of $f$ we can obtain precise bounds on the error. For example, if each cut is chosen uniformly on $(0,1)$, then the densities of the distributions of the $p_i$'s and the $(1-p_i)$'s are the same. By \cite{MN1} or Corollary A.2 of \cite{B--}, \be \E[\varphi_s(X_i)] - \log_{10} s \ \ll \ \frac1{2.9^N}, \ee and thus \be \E[P_N(s)] - \log_{10} s \ \ll \ \frac1{2^N} \sum_{i=1}^{2^N} \frac1{2.9^N} \ = \ \frac1{2.9^N}. \ee 
\end{rek}

\subsection{Variance}\label{sec:varianceunrestricted}

\begin{proof}[Proof of Theorem \ref{thm:unrestricted} (Variance)] For ease of exposition we assume all the cuts are drawn from the uniform distribution on $(0,1)$. To facilitate the minor changes needed for the general case, we argue as generally as possible for as long as possible.

We begin by noting that since $\varphi_s(X_i)$ is either 0 or 1, $\varphi_s(X_i)^2 = \varphi_s(X_i)$. From this observation, the definition of variance and the linearity of the expectation, we have
\begin{eqnarray} \label{variance algebra}
{\rm Var}\left(P_N(s)\right) & \ = \ &  \mathbb{E}[P_N(s)^2] - \mathbb{E}[P_N(s)]^2  \nonumber \\ & \ = \ & \mathbb{E}\left[ \left(\frac{\sum_{i = 1}^{2^N} \varphi_s(X_i)}{2^N}\right)^2 \right] -  \mathbb{E}[P_N(s)]^2 \nonumber \\ & \ = \ & \mathbb{E}\left[ \frac{\sum_{i = 1}^{2^N} \varphi_s(X_i)^2}{2^{2N}} + \sum_{i,j = 1 \atop i \neq j}^{2^N} \frac{\varphi_s(X_i) \varphi_s(X_j)}{2^{2N}} \right]  -  \mathbb{E}[P_N(s)]^2  \nonumber\\ & = & \frac1{2^N} \mathbb{E}\left[P_N(s)\right] + \frac{1}{2^{2N}}\left(\sum_{i, j = 1 \atop  i \neq j}^{2^N} \mathbb{E}[\varphi_s(X_i) \varphi_s(X_j)] \right) -  \mathbb{E}[P_N(s)]^2.\ \ \ \ \ \ \ \
\end{eqnarray}

From \S\ref{unrestricted expval}, $\mathbb{E}[P_N(s)] = \log_{10} s + o(1)$. Thus
\begin{eqnarray}\label{var}
{\rm Var}\left(P_N(s)\right) & \ = \ &  \frac{1}{2^{2N}}\left( \sum_{i,j = 1 \atop i \neq j}^{2^N} \mathbb{E}[\varphi_s(X_i) \varphi_s(X_j)] \right) - \log_{10}^2 s + o(1).
\end{eqnarray}

The problem is now reduced to evaluating the cross terms over all $i \ne j$. \emph{This is the hardest part of the analysis, and it is not feasible to evaluate the resulting integrals directly.} Instead, for each $i$ we partition the pairs $(X_i,X_j)$ based on how `close' $X_j$ is to $X_i$ in our tree (see Figure \ref{fig:tree3levels}). We do this as follows. Recall that each of the $2^N$ pieces is a product of the starting length $L$ and $N$ cutting proportions.  Note $X_i$ and $X_j$ must share some number of these proportions, say $k$ terms.  Then one piece has the factor $p_{k+1}$ in its product, while the other contains the factor $(1 - p_{k+1})$.  The remaining $N - k - 1$ elements in each product are independent from each other.  After re-labeling, we can thus express any  $(X_i, X_j)$ pair as
\begin{eqnarray}
X_i  & \ = \ &  L \cdot p_1 \cdot p_2 \dotsm p_k \cdot p_{k+1} \cdot p_{k+2} \dotsm p_{N} \nonumber\\ \label{Xi,Xj}
X_j & \ = \ & L \cdot p_1 \cdot p_2 \dotsm p_k \cdot (1 - p_{k+1}) \cdot \tilde p_{k+2} \dotsm \tilde p_{N}.
\end{eqnarray}

With these definitions in mind, we have
\begin{eqnarray}\label{crosstermexact}
\mathbb{E}[\varphi_s(X_i)\varphi_s(X_j)] &\ = \ &   \int_{p_1=0}^1 \int _{p_2=0}^1 \dotsm \int _{p_{N}=0}^1 \int _{\tilde p_{k+2}=0}^1 \dotsm \int _{\tilde p_{N}=0}^1 \varphi _s \left(  L \prod _{r=1}^{k+1} p_r \prod _{r = k+2} ^{N} p_r \right)  \nonumber \\
& & \ \ \ \ \cdot \  \varphi _s\left( L\prod _{r=1}^{k} p_r \cdot \left( 1-p_{k+1} \right) \cdot \prod _{r = k+2} ^{N} \tilde p_r \right) \nonumber \\
& & \ \ \ \ \cdot \  \prod _{r=1} ^{N} f (p_r)  \prod _{r=k+2} ^{N} f (1-\tilde p_r) \  dp_1 dp_2 \cdots dp_N d\tilde p_{k+2}\cdots d\tilde p_{N}.
\end{eqnarray}


The difficulty in understanding \eqref{crosstermexact} is that many variables occur in both $\varphi_s(X_i)$ and $\varphi_s(X_j)$. The key observation is that most of the time there are many variables occurring in one but not the other, which minimizes the effects of the common variables and essentially leads to evaluating $\varphi_s$ at almost independent arguments. We make this precise below, keeping track of the errors. Define
\bea
L_1 \ :=\ L\left(\prod_{r=1}^k p_r\right) p_{k+1}, \ \ \ \ \ \ L_2 \ :=\ L\left(\prod_{r=1}^k p_r\right) (1-p_{k+1}),
\eea and consider the following integrals:
\begin{eqnarray}
I(L_1) &  :=  & \int _{ p_{k+2} = 0} ^{1} \cdots \int _{ p_{N} = 0} ^{1} \varphi _s \left( L_1 \prod _{r = k+2} ^{N} p_r \right) \prod _{r = k+2} ^{N} f (p_r) \   dp_{k+2}dp_{k+3}\cdots dp_N \nonumber\\
J(L_2) & \ :=\  & \int _{ \tilde p_{k+2} = 0} ^{1} \cdots \int _{\tilde p_{N} = 0} ^{1} \varphi _s \left( L_2 \prod _{r = k+2} ^{N} \tilde p_r \right) \prod _{r = k+2} ^{N} f (\tilde p_r) \  d\tilde p_{k+2}d \tilde p_{k+3}\cdots d \tilde p_N.\nonumber\\
\end{eqnarray}

We show that, for any $L_1, L_2$, we have $|I(L_1)J(L_2)-(\log_{10} s)^2| = o(1)$. Once we have this, then all that remains is to integrate $I(L_1)J(L_2)$ over the remaining $k+1$ variables. The rest of the proof follows from counting, for a given $X_i$, how many $X_j$'s lead to a given $k$.

It is at this point where we require the assumption about $f(x)$ from the statement of the theorem, namely that $f(x)$ and $f(1-x)$ satisfy \eqref{eq:summeltransformsnozero}.  For illustrative purposes, we assume that each cut $p$ is drawn from a uniform distribution, meaning $f(x)$ and $f(1-x)$ are the probability density functions associated with the uniform distribution on $(0,1)$.  The argument can readily be generalized to other distributions; we choose to highlight the uniform case as it is simpler, important, and we can obtain a very explicit, good bound on the error.

Both $I(L_1)$ and $J(L_2)$ involve integrals over $N-k-1$ variables; we set $n := N - k - 1$.  For the case of a uniform distribution, equation (3.7) of \cite{JKKKM} (or see Corollary A.2 in \cite{B--}) gives for $n \ge 4$ that\footnote{Our situation is slightly different as we multiply the product by $L_1$; however, all this does is translate the distribution of the logarithms by a fixed amount, and hence the error bounds are preserved.}
\begin{equation}
\left|I(L_1) - \log_{10} s\right|  \ < \ \left( \frac{1}{2.9^n} +  \frac{ \zeta (n) - 1}{2.7^n} \right) 2 \log_{10} s,
\end{equation}
where $\zeta(s)$ is the Riemann zeta function, which for $\text{Re}(s)>1$ equals $\sum_{n=1}^{\infty} 1 / n^s$.
Note that for all choices of $L_1$, $I(L_1) \in [0,1)$, and for $n \le 4$ we may simply bound the difference by 1.  It is also important to note that for $n > 1$, $\zeta(n) - 1$ is  $O\left(1/2^n\right)$, and thus the error term decays very rapidly.

A similar bound exists for $J(L_2)$, and we can choose a constant $C$ such that
\begin{eqnarray}
 | I(L_1) \ - \ \log_{10} s | \ \leq\  \frac{C}{2.9^n}, \ \ \ \ \ \  | J(L_2) \ - \ \log_{10} s | \ \le \  \frac{C}{2.9^n}
\end{eqnarray}
for all $n, L_1, L_2$. Because of this rapid decay, by the triangle inequality it follows that
\begin{equation}
\left|I(L_1) \cdot J(L_2) \ - \ (\log_{10} s)^2\right| \ \leq \   \frac{2C}{2.9^n}.
\end{equation}

For each of the $2^N$ choices of $i$, and for each $1 \leq n \leq N$, there are $2^{n-1}$ choices of $j$ such that $X_j$ has exactly $n$ factors not in common with $X_i$.  We can therefore obtain an upper bound for the sum of the expectation cross terms by summing the bound obtained for $2^{n-1}I(L_1) \cdot J(L_2)$ over all $n$ and all $i$:
\begin{eqnarray}
\left|\sum_{i,j=1 \atop i \neq j}^{2^N} \left(\mathbb{E}[\varphi_s(X_i)\varphi_s(X_j)] - \log_{10}^2 s\right)\right| & \ \le \ & \sum_{i=1}^{2^N}\sum_{n = 1}^{N} 2^{n-1} \frac{2C}{2.9^n}  \ \le \ 2^N \cdot 4C.
\end{eqnarray}

Substituting this into equation \eqref{var} yields
\begin{eqnarray}
{\rm Var}\left(P_N(s)\right) &  \ \leq \ &  \frac{4C}{2^N} + o(1).
\end{eqnarray}

Since the variance must be non-negative by definition, it follows that $\lim _{N \to \infty} {\rm Var}\left(P_N(s)\right) = 0$, completing the proof if each cut is drawn from a uniform distribution. The more general case follows analogously, appealing to \cite{MN1} (or Theorem A.1 of \cite{B--}).\end{proof}

\section{Proof of Theorem \ref{thm:discrete}: Discrete Decomposition} \label{sec:discrete}
The main idea of the proof is to show that fragments generated for sufficiently large rods can be well approximated by a corresponding continuous fragmentation process in a way that preserves Benford behavior. We formalize this notion in the following lemma.

\begin{lemma}\label{approx}
Suppose that the random integer on $[1,\ell_k]$ is generated first by selecting a random real number $c_k \in [0,1]$, then rounding it up to the nearest multiple of $1/\ell_k$. Let $\mathcal{Q}$ denote the continuous process in which we start with a piece of length $\ell$, and fracture it in each iteration at $c_k$.

Let $X_k$ denote the $k$\textsuperscript{th} fragment generated $\mathcal{P}$, and $Y_k$ denote the $k$\textsuperscript{th} fragment generated by $\mathcal{Q}$. Then
\begin{equation}\label{eqapprox}
Y_k\ \leq\ X_k \left(1+\prod_{j=1}^{k}\left(1+\frac{1}{\ell_{j}}\right)\right)+O(1),
\end{equation}
where $\ell_j$ denote the remaining length of the fragment in process $\mathcal{P}$ after $j$ breaks.
\end{lemma}

In practice, we use the following corollary of this lemma.

\begin{cor}\label{corap}  Suppose $\ell_{k-1},X_k>\log^2(\ell)$ then for $g(\ell)=o(\sqrt{\log(\ell)})$, such that $g(\ell)$ tends to infinity with $\ell$ and for $k$ such that  $k<g(\ell)\log(\ell),$
 \be X_k\ = \ Y_k(1+o(1)).\ee
 \end{cor}
 We first show the corollary contingent on Lemma \ref{approx}.
 \begin{proof}[Proof of Corollary \ref{corap}]
 Using that the $\ell_k$ are monotonically decreasing, we can tightly bound how close $X_k$ and $Y_k$ are:
\be
 Y_k\ \leq\ X_k\ \ll\ Y_k \prod_{j=1}^{k}\left(1+\frac{1}{\ell_{j}}\right)+O(1). \\
 \ee
 Here we bound the terms uniformly above by the largest term in our assumptions, raised to the maximum number of fragments allowed in our hypothesis giving,
 \begin{align*}
 X_k &\ \ll\ Y_k\left(1+\frac{1}{\log^2(\ell)}\right)^{g(\ell)\log(\ell)}+O(1) \\
 &\ \ll\ Y_k e^{\frac{g(\ell)}{\log(\ell)}} +O(1),
 \end{align*}
where the second line follows by the limit definition of $e$ since we are asymptotic in $\ell$.  Taking the limit in $\ell$ and combining error terms gives the desired bound.
  \end{proof}

We now turn to the lemma.

\begin{proof}[Proof of Lemma \ref{approx}] Let $h_k$ denote the length of the fragments in process $\mathcal{Q}$ after $k$ breaks, $c_k$ denote the continuous value on chosen on $[1,h_k)$ and $d_k$ denote the rounded version of $c_k$ used in process $\mathcal{P}$.  Then we have $X_{k+1}=\ell_{k}(1-d_{k})$, $Y_{k+1}=h_{k}(1-c_{k})$ and $d_{k}<c_k<d_{k}+\frac{1}{\ell_{k-1}}$.

To prove the lemma, it suffices to show that $\ell_k\leq h_k\prod_{j=1}^{k-1}(1+\frac{1}{\ell_{j}})$, and then absorb the error from the final cut into the $O(1)$ term. Note
\be
h_k\ \leq\ \ell_k\ = \ \ell \prod_{i=1}^k d_i\ \leq\ \ell \prod_{i=1}^k \left(c_i+ \frac{1}{\ell_{i-1}}\right).
\ee
Then since $\ell_{k+1}=d_k\ell_k<c_k\ell_k$, factoring out $c_i$ from each term yields
\be\ell \prod_{i=1}^k \left(c_i+ \frac{1}{\ell_{i-1}}\right)\ \leq\ \ell \prod_{i=1}^k c_i \prod_{i=1}^k \left(1+ \frac{1}{\ell_{i}}\right)
\ \leq\ h_k \prod_{i=1}^k \left(1+ \frac{1}{\ell_{i}}\right).\ee
 \end{proof}

We now note that if we can show that almost all of the fragments satisfy the properties laid out in Corollary \ref{corap}, this will complete the proof of Theorem \ref{thm:discrete}, as the pieces generated by $\mathcal{Q}$ are strong Benford distributed by Theorem \ref{thm:unrestricted}.

\begin{lemma}\label{strongben}
Suppose $\{Y_\ell\}_\ell=\{Y_1,\dots, Y_{k_\ell}\}$ is strong Benford as $\ell$ tends to infinity. Then any set $\{X_\ell\}_\ell=\{X_1,\dots,X_{k_\ell}\}$ such that $X_{i}=Y_i(1+o(1))$ is strong Benford as $\ell$ tends to infinity.
\end{lemma}

\begin{proof}
We fix a digit $j$ and show that the distribution of the $j$\textsuperscript{th} digit is Benford. Let $D_j(a)$ denote the $j$\textsuperscript{th} digit of $a$. Then since $X_{i}=Y_i(1+o(\ell))$, there exists some function of $\ell$ tending to infinity with $\ell$ such $D_j(X_i)= D_j(Y_i)$, unless $D_{j-k}(Y_i)=9$ or $D_{j-k}(X_i)=9$ for all $k\leq f(\ell)$. Since the $Y_i$ are strong Benford and are therefore Benford in each $(j-k)$\textsuperscript{th} digit, this occurs a vanishing percentage of the time.
\end{proof}

We thus turn our attention to estimating the number of fragments generated by a piece of length $\ell$. The two following results give the necessary approximations.

\begin{lemma}\label{numpieces}
Let $F_\ell$ denote the number of fragments generated by a piece of length $\ell$. Then as $\ell \to \infty$,
 \be\mathbb{P}\left((\log(\log(\ell)))^2<F_\ell<\log(\ell)g(\ell)\right) \ = \ 1-o(1).\ee
\end{lemma}

\begin{cor}\label{dense}
Define $\ell_k$ to be the length of the rod after $k$ iterations of the process. Consider $Y':=\{y_k: \ell_k \gg \log^3(\ell)\}$. Then with probability tending to $1$,
\be \lim_{\ell\to\infty}\frac{|Y'|}{|Y|} \ = \ 1.\ee
\end{cor}
We show how Corollary \ref{dense} follows from \ref{numpieces}.

\begin{proof}[Proof of Corollary \ref{dense}]
From the upper bound on $F_\ell$, $|X\backslash X'| \ll \log(\log^3(\ell))g(\log^3(\ell))$. Thus employing our lower bound on $Y$ directly,
\be\frac{|X\backslash X'|}{|X|}\ \ll\ \frac{\log(\log(\ell))g(\log^3(\ell))}{(\log(\log(\ell)))^2}\ \ll\ \frac{g(\log^3(\ell))}{\log(\log(\ell))}.\ee
Taking $g$ such that $g(u)=o(\log(u))$ completes the proof.
\end{proof}

All that remains is to prove Lemma \ref{numpieces}.

\begin{proof}[Proof of Lemma \ref{numpieces}]
We first prove the upper bound.  We do this by finding the expected value of $F_\ell$ and applying Markov's inequality. We take the inductive hypothesis that $\mathbb{E}[F_1]=\sum_{j\leq \ell}\frac{1}{j}$, with the base case $\mathbb{E}[F_1]=1$ being clear. Since in the first break we get a single piece of length $\ell-s$ for some $s>0$, we have the recurrence relation
\be \mathbb{E}[F_\ell] \ =\ \frac{1}{\ell}+\frac{1}{\ell}\sum_{s<\ell}(1+\mathbb{E}[F_s]).\ee
Therefore, by the inductive hypothesis, we have that
\be \mathbb{E}[F_\ell]\ =\ \frac{1}{\ell}+\frac{1}{\ell}\sum_{s=1}^{\ell-1}\left(1+\sum\limits_{i=1}^s\frac{1}{i}\right).\ee

Each summand of the form $1/i$ appears in $\ell-i$ of the sums, so that
\be\frac{1}{\ell}+\frac{1}{\ell}\sum\limits_{s=1}^{\ell-1}\left(1+\sum\limits_{i=1}^s\frac{1}{i}\right)
\ = \ \frac{1}{\ell}+\frac{1}{\ell}\sum\limits_{i=1}^{\ell-1}\frac{i+(\ell-i)}{i} \ = \ \sum\limits_{i=1}^{\ell}\frac{1}{i}\ \sim\ \log(\ell)+O(1). \ee
By Markov's inequality, $P(F_\ell>\log(\ell)g(\ell))=O\left(\frac{1}{g(\ell))}\right)$.

We now prove the lower bound. The probability any of $\log(\log(\ell))^2$ breaks of a piece of length greater than $\ell^{1/2}$ at $1/(\log(\ell))$ of its original length is $o(1)$, since
\be\lim\limits_{\ell \to\infty}\left(1-\frac{1}{\log(\sqrt(\ell))}\right)^{\log(\log(\ell))^2}\ =\ \lim\limits_{\ell \to\infty} e^\frac{2\log(\log(\ell))^2}{\log(\ell)} \ = \ 1. \ee
For $\ell$ sufficiently large, if $\log(\log(\ell))^2$ breaks happen, none of which cut a piece down by a factor of $1/\log(\ell)$, then the piece is of length greater than $\ell^\frac{1}{2}.$ This follows from,
\be
\ell \cdot \left(\frac{1}{\log(\ell)}\right)^{\log(\log(\ell))^2} \ \gg\ \ell^\frac{1}{2}.
\ee
\end{proof}

\begin{proof}[Proof of Theorem \ref{thm:discrete}] The theorem follows by considering the subset $X'$, defined earlier as the set of all $X_k$ such that $\ell_k\geq \log^3(\ell)$, which is strong Benford by Lemma \ref{strongben} and Corollary \ref{corap}. Since with probability tending to one this $X'$ is almost all of $X$ (that is as $\ell \to \infty,$ we have $\frac{|X'|}{|X|}\to 1$), $X$ also exhibits strong Benford behavior asymptotically in $\ell$ with probability tending toward $1$. \end{proof}


\section{Proof of Theorem \ref{restricted thm}: Restricted 1-Dimensional Decomposition} \label{sec:restricted}

As the proof is similar to that of Theorem \ref{thm:unrestricted}, we just highlight the differences below.

\begin{proof}[Proof of Theorem \ref{restricted thm}] We may assume that $L$ = 1 as scaling does not affect Benford behavior. In the analysis below we may ignore the contributions of $X_i$ for $i \in [1,\log N]$ and all pairs of $X_i, X_j$ such that $X_i$ and $X_j$ do not differ by at least $\log N$ proportions. Removing these terms does not affect whether or not the resulting stick lengths converge to Benford (because $\log N/N \to 0$ as $N\to\infty$), but does eliminate strong dependencies or cases with very few products, both of which complicate our analysis.

As before, to prove that the stick lengths tend to Benford as $N \to \infty$ we show that $\E[P_s(N)]\to\log_{10} s$ and $\var{P_s(N)}\to 0$. The first follows identically as in \S\ref{unrestricted expval}. We have
\bea
\E[\varphi_s(x_i)] \ = \  \int_{p_1=0}^1\int_{p_2=0}^1\cdots\int_{p_i=0}^1 \varphi\left((1-p_i)\prod_{l=1}^{i-1}p_l\right) f(1-p_i)\prod_{l=1}^{i-1}f(p_l)dp_1 dp_2 \cdots dp_i \eea
tends to $\log_{10} s+o(1)$ as $n \to \infty$.

For the variance, we now have $N$ and not $2^N$ pieces, and find
\bea\label{eq:var}
\var{P_N(s)} \ = \   \frac{\log_{10}s}{N^2} + \frac{1}{N^2} \left(\sum\limits_{\substack{i, j =1 \\ i \neq j}}^{N}\E [\varphi_s(X_i) \varphi_s(X_j)] \right) - \log_{10}^2 s+o(1).
\eea Note that we may replace the above sum with twice the sum over $i<j$. Further, let $\mathcal{A}$ be the set \be \mathcal{A} \ := \ \{(i ,j) : \log N \leq i \leq j-\log N \leq N - \log N\}. \ee We may replace the sum in \eqref{eq:var} with twice the sum over pairs in $\mathcal{A}$, as the contribution from the other pairs is $o(1)$. The analysis is thus reduced to bounding
\bea
2\sum_{(i,j) \in \mathcal{A}} \E [\varphi_s(X_i) \varphi_s(X_j)],
\eea
where
\bea
X_i &\ = \ &  p_i \prod_{r = 1}^{i-1} (1-p_r)\nonumber \\
X_j &\ = \ &  p_j \prod_{r = 1}^{i-1} (1-p_r) \prod_{r =i}^{j-1} (1-p_r);
\eea
we write $X_i$ and $X_j$ in this manner to highlight the terms they have in common. Letting $f(x)$ be the density for the uniform distribution on the unit interval, we have
\bea
\E [\varphi_s(X_i) \varphi_s(X_j)] &\ = \ &  \int_{p_1 =0}^1 \cdots \int_{p_i=0}^1 \cdots \int_{p_j=0}^1 \varphi_s\left(p_i \prod_{r = 1}^{i-1} (1-p_r)\right) \nonumber\\ & & \ \ \ \ \ \cdot \varphi_s\left(p_j \prod_{r=1}^{i-1}(1-p_r)\prod_{r = i}^{j-1} (1-p_r)\right)  \cdot  \prod_{r = 1}^{j}f (p_r)  dp_1\cdots dp_j.\ \ \
\eea

The analysis of this integral is similar to that in the previous section. Let \be \mathcal{L}_i\ :=\ \prod_{r=1}^{i-1} (1-p_r), \ \ \ \ \ \ \mathcal{L}_j = \prod_{r=i+1}^{j-1}(1-p_r).\ee That is, $\mathcal{L}_i$ consists of the terms shared by $X_i$ and $X_j$, and $\mathcal{L}_j$ is the product of the terms only in $X_j$. We are left with showing that the integral \be \int_{p_1=0}^\infty \cdots \int_{p_j=0}^\infty \varphi_s\left(\mathcal{L}_i p_i\right) \varphi_s\left(\mathcal{L}_i (1-p_i) \mathcal{L}_j p_j\right) \cdot  \prod_{r = 1}^{j}f (p_r)  dp_1\cdots dp_j \ee is close to $\log_{10}^2 s$.

We highlight the differences from the previous section. The complication is that here $\mathcal{L}_1$ appears in both arguments, while before it only occurred once. This is why we restricted our pairs $(i,j)$ to lie in $\mathcal{A}$. Since we assume $i \ge \log N$, there are a lot of terms in the product of $\mathcal{L}_1$, and by the results of \cite{JKKKM} (or see Appendix A of \cite{B--}) the distribution of $\mathcal{L}_1$ converges to Benford. Similarly, there are at least $\log N$ new terms in the product for $\mathcal{L}_2$, and thus $\mathcal{L}_i (1-p_i) \mathcal{L}_j p_j$ converges to Benford. An analysis similar to that in \S\ref{sec:unrestricted} shows that the integral is close to $\log_{10}^2 s$ as desired. The proof is completed by noting that the cardinality of $\mathcal{A}$ is $N^2/2 + O(N \log N)$. Substituting our results into \eqref{eq:var}, we see the variance tends to 0. Thus the distribution of the leading digits converges in distribution to Benford's Law.
\end{proof}


\section{Proof of Theorem \ref{fixed proportion thm}: Fixed 1-Dimensional Proportion Decomposition} \label{sec:fixedprop}
Recall that we are studying the distribution of the stick lengths that result from cutting a stick at a fixed proportion $p$. We define $y$ by $10^y := \frac{1-p}{p}$, the ratio between adjacent piece lengths. The resulting behavior is controlled by the rationality of $y$. We see this clearly in the three examples in Figures \ref{fig:rational311} through \ref{fig:notrational113310}, where we show observed behavior plotted against Benford behavior.

\begin{figure}[h]
\begin{center}
\scalebox{.75}{\includegraphics{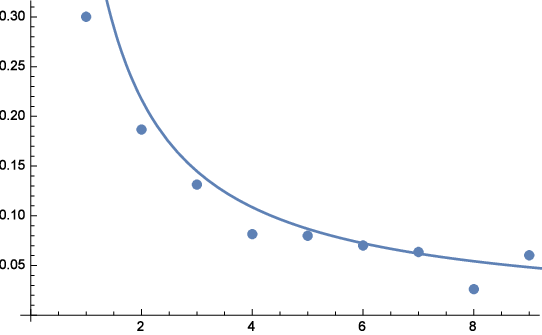}}\
\scalebox{.75}{\includegraphics{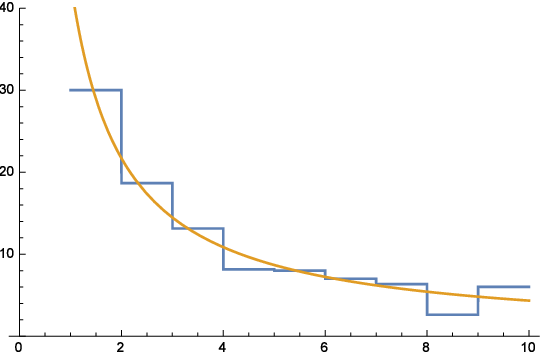}}
\caption{\label{fig:rational311} Irrational case: $p = 3/11$, 1000 levels; $y = \log_{10}(8/3) \not\in \mathbb{Q}$.}
\end{center}\end{figure}

\begin{figure}[h]
\begin{center}
\scalebox{.75}{\includegraphics{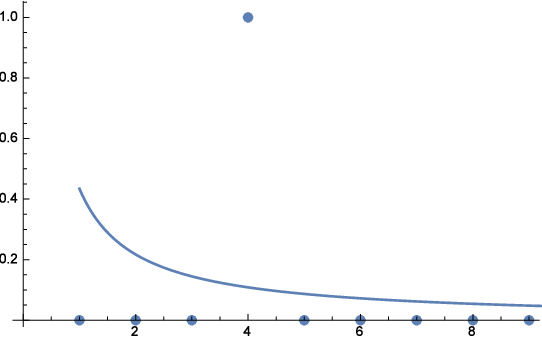}}\ \scalebox{.75}{\includegraphics{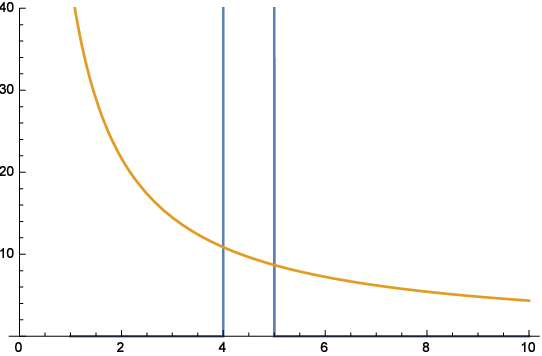}}
\caption{\label{fig:notrational111} Rational case: $p = 1/11$, 1000 levels; $y = 1 \in \mathbb{Q}$.}
\end{center}\end{figure}

\begin{figure}[h]
\begin{center}
\scalebox{.75}{\includegraphics{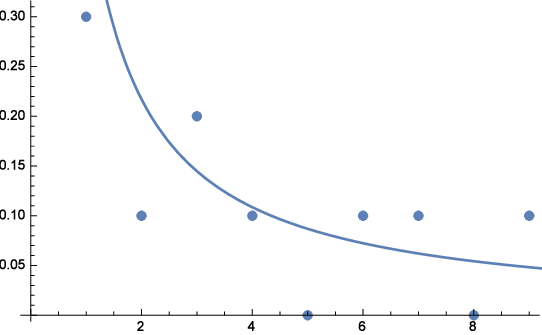}}\ \scalebox{.75}{\includegraphics{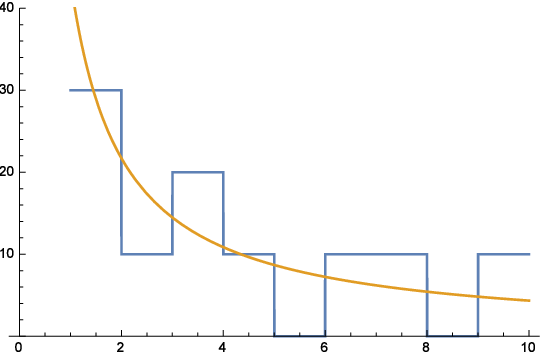}}
\caption{\label{fig:notrational113310} Rational case: $p = 1/(1 + 10^{33/10})$, 1000 levels; $y = 33/10 \in \mathbb{Q}$.}
\end{center}\end{figure}


\subsection{Case I: $y\in\Q$.} \label{y rational}

Let $y=r/q$. Here $r\in\Z, q\in\N$ and $\gcd(r,q)=1$. Let $S_{10}(x_j)$ denote the first digit of $x_j$. As
\bea
x_{j+q} \ = \   \left(\frac{1-p}{p}\right)^{q}x_j \ = \  (10^y)^q x_j \ = \  10^r x_j, \eea it follows that \bea
S_{10}(x_{j+q})&\ = \ & S_{10}(x_j).
\eea
Thus the significand of $x_j$ repeats every $q$ indices.\footnote{We are interested in determining the frequency with which each leading digit occurs. It is possible that two sticks $x_j$ and $x_i$ are not a multiple of $q$ indices apart but still have the same leading digit. Thus summing the frequency of every $q$\textsuperscript{th} length tells us that for each digit $d$ the probability of a first digit $d$ is $a /q$ for some $a \in \N$.} We now show that the $q$ different classes of leading digits occur equally often as $N\to\infty$.

To do this, we use the multisection formula. Given a power series $f(x) = \sum_{k = 0}^{\infty} a_{k} x^{k}$, we can take the multisection $\sum_{\ell = 0}^{\infty} a_{\ell q + j} x^{\ell q + j}$, where $j$ and $q$ are integers with \ $0 \leq j < q$. The multisection itself is a power series that picks out every $q$\textsuperscript{th} term from the original series, starting from the $j$\textsuperscript{th} term. We have a closed expression for the multisection in terms of the original function (see \cite{Che} for a proof of this formula):
\bea\sum_{\ell = 0}^{\infty} a_{\ell q + j} x^{\ell q + j} &\ = \ &  \frac{1}{q} \sum_{s = 0}^{q - 1} \omega^{- j s} f(\omega^s x), \eea
where $\omega=e^{2 \pi i / q}$ is a primitive $q$\textsuperscript{th} root of unity. We apply this to $f(x) = (1 + x)^{N} = \sum_{k = 0}^{N} {N \choose k} x^k$. To extract the sum of equally spaced binomial coefficients, we take the multisection of the binomial theorem with $x = 1$:
\bea\sum_{\ell} {N \choose \ell q + j} &\ = \ &  \frac{2^{N}}{q}  \sum_{s = 0}^{q - 1} \left( \cos \frac{\pi s}{q} \right) ^{N} \cos \frac{\pi (N- 2 j) s}{q};
\eea note in the algebraic simplifications we took the real part of $\omega^{(N-2j)/2}$, which is permissible as the left hand side is real and therefore the imaginary part sums to zero.

All terms with index $j$ mod $q$ share the same leading digit. Therefore the probability of observing a term with index $j \bmod q$ is given by
\bea\label{eq:multisectionjmodq}
\frac{1}{2^N}\left[{N \choose j}+{N \choose j+q}+{N\choose j+2q}+\dotsb\right] &\ = \ &  \frac{1}{q}\sum_{s=0}^{q-1} \left(\cos \frac{\pi s }{q}\right)^N\cos\frac{\pi(N-2j)s}{q} \nonumber\\
&\ = \ & \frac{1}{q}\left(1+\sum_{s=1}^{q-1} \left(\cos \frac{\pi s }{q}\right)^N\cos\frac{\pi(N-2j)s}{q}\right) \nonumber\\
&\ = \ & \frac{1}{q}\left(1 + {\rm Err} \left[ (q - 1) \left( \cos \frac{\pi}{q}\right)^N\right] \right),
\eea
where ${\rm Err}[X]$ indicates an absolute error of size at most $X$. When $q = 1$, the term inside the ${\rm Err}$ vanishes. For $q \in \N$, $q > 1$, $\cos (\pi/q) \in [0,1)$; as that value is raised to the $N$\textsuperscript{th} power, it approaches 0 exponentially fast. As $N\to\infty$, the term inside the ${\rm Err}$ disappears, leaving us $1/q$. Hence the probability of observing a particular leading digit converges to a multiple of $1/q$, which is a rational number. On the other hand, the probability from the Benford distribution is $\log_{10} (1+1/d)$ which is an irrational number. Therefore the described cutting process does not result in perfect Benford behavior. \hfill $\Box$

\begin{rek} Instead of using the multisection formula, we could use the monotonicity (as we move towards the middle) to show that the different classes of $j \bmod q$ have approximately the same probability by adding or removing the first and/or last term in the sequence, which changes which class dominates the other. We chose this approach as the multisection formula is useful in the proof of Theorem \ref{fixed proportion thm} when the irrationality exponent of $y$ is finite.\end{rek}

\subsection{Case II: $y\notin\Q$ has finite irrationality exponent}\label{y irrational} We prove the leading digits of the $2^N$ stick lengths are Benford by showing that the logarithms of the piece lengths are equidistributed modulo 1 (Benford's Law then follows by simple exponentiation; see \cite{Dia,MT-B}). The frequency of the lengths $x_j$ follow a binomial distribution with mean $N /2$ and standard deviation $\sqrt{N}/2$. As the argument is long we briefly outline it. First we show that the contributions from the tails of the binomial distribution are small. We then break the binomial distribution into intervals that are a power of $N$ smaller than the standard deviation, and show both that the probability density function does not change much in each interval and that the logarithms of the lengths in each interval are equidistributed modulo 1.

Specifically, choose a $\delta \in (0, 1/2)$; the actual value depends on optimizing various errors. Note that $N^\delta \ll \sqrt{N}/2$, the standard deviation. Let \be x_\ell \ := \ \frac{N}{2}+ \ell N^\delta, \ \ \ \ \ \  x_{\ell,i} \  = \ \frac{N}{2}+ \ell N^\delta + i, \ \ \ \ \ \ I_\ell \ :=\ \{x_\ell, x_\ell+1, \dots, x_\ell+N^\delta-1\}.\ee There are $N/N^\delta=N^{1-\delta}$ such intervals. By symmetry, it suffices to just study the right half of the binomial.


\subsubsection{Truncation} \label{truncation}

Instead of considering the entire binomial distribution, for any $\epsilon > 0$ we show that we may truncate the distribution and examine only the portion that is within $N^\epsilon$ standard deviations of the mean. Recall that we are only considering the right half of the binomial as well.

For $\epsilon>0$, Chebyshev's Inequality\footnote{While we could get better bounds by appealing to the Central Limit Theorem, Chebyshev's Inequality suffices.} gives that the proportion of the density that is beyond $N^\epsilon$ standard deviations of the mean is
\bea
{\rm Prob}\left(\left |x-\frac{N}{2} \right |\ge N^\epsilon N^{1/2}/2\right)\ \le \ \frac{1}{N^{2\epsilon}}.
\eea
As $N$ tends to infinity this probability becomes negligible, and thus we are justified in only considering the portion of the binomial from $\frac{N}{2}- N^{\frac{1}{2} + \epsilon}$ to $\frac{N}{2}+ N^{\frac{1}{2} + \epsilon}$. Thus $\ell$ ranges from $-N^{\frac{1}{2}-\delta+\epsilon}$ to $N^{\frac{1}{2}-\delta+\epsilon}$.


\subsubsection{Roughly Equal Probability Within Intervals}\label{Equal Probability}

Let $x_{\ell}=N/2+\ell N^\delta$. Consider the difference in the binomial coefficients of adjacent intervals, which is related to the difference in probabilities by a factor of $1/2^N$. Note that this is a bound for the maximum change in probabilities in an interval of length $N^\delta$ away from the tails of the distribution. For future summation, we want to relate the difference to a small multiple of either endpoint probability; it is this restriction that necessitated the truncation from the previous subsection. Without loss of generality we may assume $\ell \ge 0$ and we find
\bea
{N \choose x_{\ell}}-{N \choose x_{\ell+1}}&\ = \ & {N \choose \frac{N}{2}+\ell N^\delta}-{N \choose \frac{N}{2}+(\ell+1)N^\delta}\nonumber\\
&\ = \ & \frac{N!}{(\frac{N}{2}+\ell N^\delta)!(\frac{N}{2}-\ell N^\delta)!}-\frac{N!}{(\frac{N}{2}+(\ell+1)N^\delta)!(\frac{N}{2}-(\ell+1)N^\delta)!}\nonumber\\
&\ = \ & \frac{N!(\frac{N}{2}+(\ell+1)N^\delta)!(\frac{N}{2}-(\ell+1)N^\delta)!-N!(\frac{N}{2}+\ell N^\delta)!(\frac{N}{2}-\ell N^\delta)!}{(\frac{N}{2}+\ell N^\delta)!(\frac{N}{2}-\ell N^\delta)!(\frac{N}{2}+(\ell+1)N^\delta)!(\frac{N}{2}-(\ell+1)N^\delta)!}\nonumber\\
&\ = \ & {N \choose \frac{N}{2}+\ell N^\delta}\left[1-\frac{(\frac{N}{2}+\ell N^\delta)!(\frac{N}{2}-\ell N^\delta)!}{(\frac{N}{2}+(\ell+1)N^\delta)!(\frac{N}{2}-(\ell+1)N^\delta)!}\right]\nonumber\\
&\ = \ & {N \choose x_{\ell}}\left[1-\frac{(\frac{N}{2}+\ell N^\delta)!(\frac{N}{2}-\ell N^\delta)!}{(\frac{N}{2}+(\ell+1)N^\delta)!(\frac{N}{2}-(\ell+1)N^\delta)!}\right] \label{eq:equalprob}.
\eea
Notice here that the difference in binomial coefficients is in terms of the probability at the left endpoint of the interval, which allows us to express the difference in probabilities relative to the probability within an interval. Let $\alpha_{\ell;N}=\frac{(\frac{N}{2}+\ell N^\delta)!(\frac{N}{2}-\ell N^\delta)!}{(\frac{N}{2}+(\ell+1)N^\delta)!(\frac{N}{2}-(\ell+1)N^\delta)!}$. We show that $1-\alpha_{\ell;N}\to 0$, which implies the probabilities do not change significantly over an interval. We have

\bea
\alpha_{\ell;N}&\ = \ & \frac{(\frac{N}{2}+\ell N^\delta)!(\frac{N}{2}-\ell N^\delta)!}{(\frac{N}{2}+(\ell+1)N^\delta)!(\frac{N}{2}-(\ell+1)N^\delta)!}\nonumber\\
&\ge&\frac{(\frac{N}{2}-(\ell+1)N^\delta)^{N^\delta}}{(\frac{N}{2}+(\ell+1)N^\delta)^{N^\delta}}\nonumber\\
&\ = \ & \left(\frac{1-\frac{2(\ell+1)}{N^{1-\delta}}}{1+\frac{2(\ell+1)}{N^{1-\delta}}}\right)^{N^\delta}\nonumber\\
\log\alpha_{\ell;N}&\ge&N^\delta\left[\log\left(1-\frac{2(\ell+1)}{N^{1-\delta}}\right)-\log\left(1+\frac{2(\ell+1)}{N^{1-\delta}}\right)\right].
\eea
From Taylor expanding we know $\log(1+u) = -u + u^2/2 - O(u^3)$. Thus letting $u = 2(\ell+1)/N^{1-\delta}$ (which is much less than 1 for $N$ large as $\ell \le N^{\frac12-\delta+\epsilon}$) in the difference of logarithms above we see the linear terms reinforce and the quadratic terms cancel, and thus the error is of size $O(u^3) = O(\ell^3 / N^{3-3\delta}) = O(N^{3\epsilon-3/2})$. Therefore
\bea
\log \alpha_{\ell;N}&\ge&N^\delta\left[-\frac{4(\ell+1)}{N^{1-\delta}}+O(N^{3\epsilon-3/2})\right]\nonumber\\
\alpha_{\ell;N}&\ge&e^{-4(\ell+1)N^{2\delta-1}+O(N^{\delta+3\epsilon-3/2})}\nonumber\\
&\ge&1-4(\ell+1)N^{2\delta-1}+O(N^{\delta+3\epsilon-3/2})\nonumber\\
1-\alpha_{\ell;N}&\ \le \ & 4(\ell+1)N^{2\delta-1}+O(N^{\delta+3\epsilon-3/2}) \label{eq:alpha}.
\eea

Since we have truncated and $\ell \le N^{\frac{1}{2}-\delta+\epsilon}$, this implies $(\ell+1)N^{2\delta -1} \ll N^{\delta+\epsilon - \frac12}$, which tends to zero if $\delta < 1/2 - \epsilon$. Substituting \eqref{eq:alpha} into \eqref{eq:equalprob} yields
\bea
{N \choose x_{\ell}}-{N \choose x_{\ell+1}}&\ = \ & {N \choose \frac{N}{2}+\ell N^\delta}(1-\alpha_{\ell;N})\nonumber\\
&\ \le \ & {N \choose x_{\ell}} \left(4(\ell+1)N^{2\delta-1}+O(N^{\delta+3\epsilon-3/2})\right).
\label{l bound}
\eea
Since $\ell\le N^{1/2-\delta+\epsilon}$, it follows that
\bea
\bigg|{N \choose x_{\ell}}-{N \choose x_{\ell+1}}\bigg| &\ \ll \ &  {N \choose x_{\ell}}\left(N^{-\frac{1}{2}+\delta+\epsilon}+O(N^{\delta+3\epsilon-3/2})\right).
\eea
As $\delta<1/2-\epsilon$, $O(N^{\delta+3\epsilon-3/2})$ is dominated by $N^{-\frac{1}{2}+\delta+\epsilon}$ since $\epsilon$ is small. We have proved

\begin{lem}\label{lem:diffxlxlplusone} Let $\delta \in (0, 1/2 - \epsilon)$ and $\ell \le N^{\frac{1}{2}-\delta+\epsilon}$. Then for any $i \in \{0, 1, \dots, N^\delta\}$ we have
\bea
\left|{N \choose x_{\ell}}-{N \choose x_{\ell,i}}\right|&\ \ll \ &  {N\choose x_{\ell}} N^{-\frac{1}{2}+\delta+\epsilon}.
\eea
\end{lem}


\subsubsection{Equidistribution}\label{Equidistribution} We first do a general analysis of equidistribution of a sequence related to our original one, and then show how this proves our main result.

Given an interval \be I_\ell \ :=\ \left\{x_{\ell,i}: \frac{N}{2}+\ell N^\delta, \dots,   \frac{N}{2}+(\ell+1) N^\delta -1\right\},\ee we prove that $\log (x_{\ell,i})$ becomes equidistributed modulo 1 as $N\to\infty$. Fix an $(a,b) \subset (0,1)$. Let $J_\ell(a,b)\subset \{0,1,\dots, N^\delta\}$ be the set of all $i \in I_\ell$ such that $\log(x_{\ell,i}) \bmod 1\in (a,b)$; we want to show its measure is $(b-a) N^{\delta}$ plus a small error. As the $x_{\ell,i}$ form a geometric progression with common ratio $r = \frac{1-p}{p} = 10^y$, their logarithms are equidistributed modulo 1 if and only if $\log r = y$ is irrational (see \cite{Dia,MT-B}). Moreover, we can quantify the rate of equidistribution if the irrationality exponent $\kappa$ of $y$ is finite. From \cite{KN} we obtain a power savings:
\bea
|J_\ell(a,b)| &\ = \ &  (b-a)N^\delta + O\left(N^{\delta(1-\frac{1}{\kappa}+\epsilon')}\right);
\eea see \cite{KonMi} for other examples of systems (such as the $3x+1$ map) where the irrationality exponent controls Benford behavior. The key idea is to keep approximating general sums with simpler ones that are tractable, with manageable errors at each step.

We now combine this quantified equidistribution with Lemma \ref{lem:diffxlxlplusone} and find (we divide by $2^N$ later, which converts these sums to probabilities)
\bea\label{eq:sumiinJab} \sum_{i \in J_\ell(a,b)} \ncr{N}{x_{\ell,i}} & \ = \ & \sum_{i \in J_\ell(a,b)} \left[\ncr{N}{x_\ell} + O\left({N\choose x_{\ell}} N^{-\frac{1}{2}+\delta+\epsilon}\right)\right] \nonumber\\ & = & \ncr{N}{x_\ell} \left[\left(\sum_{i \in J_\ell(a,b)} 1\right) + O\left( N^{-\frac{1}{2}+\delta+\epsilon} N^\delta\right)\right] \nonumber\\ & = & \ncr{N}{x_\ell} \left[\left( (b-a) N^{\delta} + O\left(N^{\delta(1-\frac{1}{\kappa}+\epsilon')}\right)\right) + O\left(N^{-\frac12+2\delta+\epsilon}\right)\right]\nonumber\\ & = & (b-a) N^\delta \ncr{N}{x_\ell} + N^\delta \ncr{N}{x_\ell} \cdot O\left(N^{-\frac12+\delta+\epsilon} + N^{-\delta(\frac{1}{\kappa}-\epsilon')}\right). \eea
Notice the error term above is a power smaller than the main term. If we show the sum over $\ell$ of the main term is $1+o(1)$, then the sum over $\ell$ of the error term is $o(1)$ and does not contribute in the limit (it will contribute $N^{-\eta}$ for some $\eta > 0$).

As $x_\ell = \frac{N}{2} + \ell N^\delta$, we use the multisection formula (see \eqref{eq:multisectionjmodq}) with $q = N^\delta$, and find \bea \sum_{\ell= -N/N^\delta}^{N/N^\delta} \ncr{N}{x_\ell} \ = \ \frac{2^N}{N^\delta}\left(1 + {\rm Err} \left[N^\delta \left( \cos \frac{\pi}{N^\delta}\right)^N\right] \right) \ = \ \frac{2^N}{N^\delta} + O\left(2^N \cdot e^{-3N^{1-2\delta}}\right),
\eea
where ${\rm Err}[X]$ indicates an absolute error of size at most $X$ and the simplification of the error comes from Taylor expanding the cosine and standard analysis: \bea \log \left(\cos(\pi/N^\delta)\right)^N & \ = \ & N \log \left(1 - \frac1{2!} \frac{\pi^2}{N^{2\delta}} + O(N^{-4\delta})\right) \nonumber\\ &=& N \left(-\frac1{2!} \frac{\pi^2}{N^{2\delta}} + O(N^{-4\delta})\right) \ = \ - \frac{\pi^2}{2} N^{1-2\delta} + O(N^{1-4\delta}) \nonumber\\ \cos(\pi/N^\delta) & \le & e^{-3N^{1-2\delta}} \eea for $N$ large.

Remember, though, that we are only supposed to sum over $|\ell| \le N^{\frac{1}{2}-\delta+\epsilon}$. The contribution from the larger $|\ell|$, however, was shown to be at most $O(N^{-2\epsilon})$ in \S\ref{truncation}, and thus we find
\bea \frac1{2^N}\sum_{\ell=-N^{\frac{1}{2}-\delta+\epsilon}}^{N^{\frac{1}{2}-\delta+\epsilon}} (b-a) N^\delta \ncr{N}{x_\ell}& \ = \ & 1 + O\left(N^\delta e^{-3N^{1-2\delta}} + N^{-2\epsilon}\right). \eea As this is of size 1, the lower order terms in \eqref{eq:sumiinJab} do not contribute to the main term (their contribution is smaller by a power of $N$).

We can now complete the proof of Theorem \ref{fixed proportion thm} when $y\not\in\Q$ has finite irrationality exponent. Convergence in distribution to Benford's law is equivalent to showing, for any $(a,b) \subset (0,1)$, that \be \sum_{\ell=-N/N^\delta}^{N/N^\delta} \sum_{i=0 \atop S_{10}(x_{\ell,i}) \in (a,b)}^{N_\delta-1} \frac{\ncr{N}{x_{\ell,i}}}{2^N} \ = \ b-a + o(1);\ee however, we just showed that. Furthermore, our analysis gives a power savings for the error, and thus we may replace the $o(1)$ with $N^{-\eta}$ for some computable $\eta > 0$ which is a function of $\epsilon, \epsilon'$ and $\delta$. This completes the proof of this case of Theorem \ref{fixed proportion thm}. \hfill $\Box$

\begin{rek} A more careful analysis allows one to optimize the error. We want $\delta\left(-\frac{1}{\kappa}+\epsilon'\right)$ $=$ $-\frac{1}{2} + \delta + \epsilon$, and thus we should take $\delta = (\frac{1}{2} - \epsilon)/(1+\frac{1}{\kappa} - \epsilon')$. Of course, if we are going to optimize the error we want a significantly better estimate for the probability in the tail. This is easily done by invoking the Central Limit Theorem instead of Chebyshev's Inequality. \end{rek}


\subsection{Case III: $y\notin\Q$ has infinite irrationality exponent} While almost all numbers have irrationality exponent at most 2, the argument in Case II does not cover all possible $y$ (for example, if $y$ is a Liouville number such as $\sum_n 10^{-n!}$). We can easily adapt our proof to cover this case, at the cost of losing a power savings in our error term. As $y$ is still irrational, we still have equidistribution for the logarithms of the segment lengths modulo 1; the difference is now we must replace $O\left(N^{\delta(1-\frac{1}{\kappa}+\epsilon')}\right)$ with $o(N^\delta)$. The rest of the argument proceeds identically, and we obtain in the end an error of $o(1)$ instead of $O(N^{-\eta})$.

%
\section{Two-Dimensional Fragmentation} \label{sec:2d}
In this section we prove Theorem \ref{thm:triangle} and Theorem \ref{thm:quadrilateral}, extending our results to two-dimensions.

\subsection{Decomposition of Triangles}

For any two triangles there is an affine transformation of the plane that maps one to the other. Recall that affine transformations preserve ratios of areas in the plane. In particular, we may consider an arbitrary triangle as an affine transformation of an initial right triangle. It therefore makes sense to select a point in the interior of each sub-triangle $T$ according to the probability distribution obtained by mapping $f$ under the affine transformation $T_0 \to T$. For example, when $f$ is uniform, our selection is uniformly random for each sub-triangle in the decomposition.

\begin{proof}[Proof of Theorem \ref{thm:triangle}] Considering at each stage of the decomposition the proportions $p^1, p^2, p^3$ of the resultant sub-triangles to their parent, we obtain a sequence of random variables $\{ (p_i^1,p_i^2,p_i^3) \}_{i \in \mathbb{N}}$. Since affine transformations preserve ratios of areas, the probability density $g$ on $(0,1)^3$ of the proportions $p_1,p_2,p_3$ is constant with respect to the shape of the triangle and depends only on $f$. We may therefore reframe this two-dimensional process in terms of our one-dimensional rod decomposition, where at each step the rod is broken into three pieces with proportions $p^1,p^2,p^3$ according to the distribution $g$.

In the case of our initial right triangle with vertices $(0,0), (0,1), (1,0)$, if we select the point $(x,y)$ in the interior of our triangle, the formula $\text{area}=\text{base} \times \text{height}$ gives that the sub-triangles have areas $\frac{1}{2}x, \frac{1}{2}y, \frac{1}{2}(1-x-y)$. The proportions of the areas of the sub-triangles to the original triangle are therefore $x,y,1-x-y$ (in particular, if $f$ is uniform, then so is $g$.) Theorem \ref{thm:triangle} now follows directly from Theorem  \ref{thm:unrestricted} under an affine mapping. \end{proof}

\subsection{Decomposition of Quadrilaterals}
We now consider our model for decomposing convex quadrilaterals.

\begin{defi}
Let
\be S^k\ :=\ \{ \mathbf{x} \in (0,1)^k: x_1 + \cdots + x_k = 1\}.\ee
\end{defi}

A key lemma is a variant of the unrestricted fragmentation process in which pieces are allowed to rejoin after breaks. \fix{DAVID/BLAINE: DEFINE WHAT THE TILDES ARE OVER $f$ AND $G$ BELOW.}
\reply{DAVID: I THINK THIS IS WHAT IS MEANT \be \tilde{f}(x)\ =\ \int_{y_1 = - \infty}^\infty \cdots \int_{y_{k-1} = - \infty}^\infty f(x,y_1,\ldots,y_{k-1}) d y_1 \cdots d y_{k-1}.\ee}
\begin{lem}\label{lem:rodjoining}
Fix a probability density $\pi$ on $(0,1)$, a probability density $f$ on $S^k$, and a one-parameter probability density $G(\theta)$ on $S^k$ depending on $\theta$ a single random variable such that \eqref{modmell} holds. Given a stick of length $L$, we perform the following iterated process.
\begin{enumerate}
\item[1.] Break the stick at a proportion $p$ chosen according to $\pi$.

\item[2.] Break one of the pieces into $k$ sub-pieces $X_1, \ldots, X_k$ with proportions chosen according to $f$.

\item[3.] Independently break the other piece into $k$ sub-pieces $Y_1, \ldots, Y_k$ with proportions according to $G(p)$.

\item[4.] Combine each piece $X_i$ with the piece $Y_i$, so that we have $k$ pieces of lengths $X_1 + Y_1, \ldots, X_k + Y_k$.

\item[5.] Repeat this process independently on each of the four resulting sub-pieces until there are $k^N$ pieces $\mathcal{X}$.
\end{enumerate}
If
\begin{equation} \label{modmell}
\lim_{N \to \infty} \sum_{\substack{ \ell = -\infty \\ \ell \neq 0}}^{\infty} \left( \int_0^\infty \int_0^1 \int_0^1 \int_0^1
	\pi(p) \tilde{f}(y) \tilde{G}(p;z) \tilde{f} \left( \frac{x-pz}{(1-p)y} \right) x^{-2 \pi i \ell / \log 10} d z d y d p d x \right)^N \ = \ 0,
\end{equation}
then as $N \to \infty$, the significands of the lengths in $\mathcal{X}$ converge in distribution to Benford's law.
\end{lem}

\begin{proof}[Proof of Lemma \ref{lem:rodjoining}]
For readability of the proof, we assume that our distributions $f,G(\theta)$ are symmetric; the general case follows similarly. Let $\varphi_s$ denote the significand indicator function, so that
\be\varphi_s(x)\ :=\ \begin{cases} 1 & \text{if the significand of $x$ is at most $s$}
	\\ 0 & \rm otherwise. \end{cases} \ee
Let $P_N(s)$ denote the proportion of pieces in $\mathcal{X}$ the significand of whose length is at most $s$. That is,
\be P_N(s)\ :=\ \frac{1}{|\mathcal{X}|} \sum_{X \in \mathcal{X}} \varphi_s(X).\ee

We now show that
\begin{equation} \label{exp}
\lim_{N \to \infty} \mathbb{E}[P_N(s)]\ =\ \log_{10}(s)
\end{equation}
and
\begin{equation} \label{var}
\lim_{N \to \infty} \var{P_N(s)} \ =\ 0,
\end{equation}
which together suffice to prove our theorem. The proof of \eqref{exp} proceeds as in the proof of Theorem \ref{thm:unrestricted}
except that we apply Theorem 1.1 of \cite{JKKKM} with $f_{D(\theta)} = 1/(\theta + p/4)$ as opposed to $f_{D(\theta)} = 1/\theta$. We therefore concentrate our attention on showing \eqref{var}.

We have
\begin{align}
\var{P_N(s)} &\ = \ \mathbb{E}[P_N^2(s)] - \mathbb{E}[P_N(s)]^2 \nonumber\\
& \ =\ \frac{1}{4^{2N}} \sum_{X,Y \in \mathcal{X}} \mathbb{E}[\varphi_s(X)\varphi_s(Y)]
	- \mathbb{E}[P_N(s)]^2 \nonumber\\
&\ =\ \frac{1}{4^N} E[P_N(s)]
	+ \frac{1}{4^{2N}} \sum_{X \neq Y \in \mathcal{X}}\mathbb{E}[\varphi_s(X)\varphi_s(Y)]
	- E[P_N(s)]^2 \nonumber\\
&\ =\ \frac{1}{4^{2N}} \sum_{X \neq Y \in \mathcal{X}} \mathbb{E}[\varphi_s(X)\varphi_s(Y)]
	- \log_{10}(s) + o(1).
\end{align}
Now, consider two arbitrary pieces $X \neq Y \in \mathcal{X}$. After re-labeling, we may express $X,Y$ as
\begin{align*}
X &\ =\ L \left( p_1\alpha_1 + (1-p_1) \beta_1 \right)
	\cdots \left(p_{k-1} \alpha_{k-1} + (1-p_{k-1})\beta_{k-1} \right)
	\left( p_k \alpha_k + (1-p_k) \beta_k \right) \nonumber\\
	& \qquad \left( p_{k+1} \alpha_{k+1} + (1-p_{k+1})\beta_{k+1} \right) \cdots
	\left( p_N \alpha_N + (1-p_N) \beta_N \right) \nonumber\\
&\ =:\ L a_1 \cdots a_{k-1} \left( p_k \alpha_k + (1-\beta_k)\beta_k \right) a_{k+1} \cdots a_N \nonumber\\
&\ =:\ L_1 a_{k+1} \cdots a_N
\end{align*}
\begin{align}
Y &\ =\ L \left( p_1\alpha_1 + (1-p_1) \beta_1 \right)
	\cdots \left(p_{k-1} \alpha_{k-1} + (1-p_{k-1})\beta_{k-1} \right)
	\left( p_k \alpha_k' + (1-p_k) \beta_k' \right) \nonumber\\
	& \qquad \left( p_{k+1} \tilde{\alpha}_{k+1} + (1-p_{k+1})\tilde{\beta}_{k+1} \right) \cdots
	\left( p_N \tilde{\alpha}_N + (1-p_N) \tilde{\beta}_N \right) \nonumber\\
&\ =:\ L a_1 \cdots a_{k-1} \left( p_k \alpha_k' + (1-p_k)\beta_k' + \frac{p}{4} \right) \tilde{a}_{k+1} \cdots \tilde{a}_N \nonumber\\
&\ =:\ L_2 \tilde{a}_{k+1} \cdots \tilde{a}_N
\end{align}
where $a_i,\tilde{a}_i$ are independent. By symmetry, we may take $\alpha_i, \tilde{\alpha}_i, \beta_i$, and $\tilde{\beta}_i$ to be the first coordinates of random variables chosen according to $f$ or $G(p_i)$. However, $\alpha_i'$ and $\beta_i'$ are not independent. The distribution of the $a_i,\tilde{a}_i$ has density
\be
h(x)\ =\ \int_0^1 \phi(p) \int_0^1 \tilde{f}(y) \int_0^1 \tilde{G}(p;z) \tilde{f}\left( \frac{x-pz}{(1-p)y} \right) d z d y d p,
\ee
where for any probability distribution $f$ in $k$ variables, we define a probability density $\tilde{f}$ by
\be \tilde{f}(x)\ =\ \int_{y_1 = - \infty}^\infty \cdots \int_{y_{k-1} = - \infty}^\infty f(x,y_1,\ldots,y_{k-1}) d y_1 \cdots d y_{k-1}.\ee
As in \S\ref{sec:unrestricted}, we write
\begin{eqnarray} I(L_1) &  \ = \ & \int_{a_{k+1} = 0}^1 \cdots \int_{a_N = 0}^1 \varphi_s\left( L_1 \prod_{i=k+1}^N a_i \right)
	\prod_{i=k+1}^N h(a_i) d a_{k+1} \cdots d a_N \nonumber\\
 I(J_2) &=& \int_{\tilde{a}_{k+1} = 0}^1 \cdots \int_{\tilde{a}_N = 0}^1
	\varphi_s\left( L_2 \prod_{i=k+1}^N \tilde{a}_i \right)
	\prod_{i=k+1}^N h(\tilde{a}_i) d \tilde{a}_{k+1} \cdots d \tilde{a}_N \end{eqnarray}
so that
\be
\mathbb{E}[\varphi_s(X)\varphi(Y)]\ =\  \int_{a_1 = 0}^1 \cdots \int_{a_{k-1}=0}^1 \int_{\alpha_k = 0}^1 \int_{\beta_k=0}^1 I(L_1) J(L_2) d a_1 \cdots d a_{k-1} d \alpha_k d \beta_k.
\ee
It suffices to show that, for any $L_1, L_2$, we have $| I(L_1) J(L_2) - (\log_{10} s)^2 | = o(1)$ in $N$. This is proven in \S\ref{sec:unrestricted} if $h$ satisfies the Mellin condition \eqref{modmell}. In other words if
\begin{equation}
\lim_{N \to \infty} \sum_{\substack{ \ell = -\infty \\ \ell \neq 0}}^{\infty} \left( \int_0^\infty \int_0^1 \int_0^1 \int_0^1
	\pi(p) \tilde{f}(y) \tilde{G}(p;z) \tilde{f} \left( \frac{x-pz}{(1-p)y} \right) x^{-2 \pi i \ell / \log 10} d z d y d p d x \right)^N\ =\ 0,
\end{equation}
which is precisely \eqref{modmell} on $\pi$,$f$,$G$.
\end{proof}

We can now proceed to the proof of Benford behavior in the quadrilateral decomposition model. \fix{DAVID / BLAINE: DO YOU THINK WE SHOULD GIVE AN EXAMPLE OF SUCH A MAP? MAYBE AS AN APPENDIX?}
\reply{DAVID: I THINK THAT THIS MAY BE OUTSIDE THE SCOPE OF WHAT WE WANT TO DO, THE MAPPINGS I HAVE SEEN AREN'T THAT SUCCINCTLY DESCRIBED}
\begin{proof}[Proof of Theorem \ref{thm:quadrilateral}]
Unlike in the triangular case, affine transformations do not map between arbitrary quadrilaterals, since affine transformations must preserve parallelism. However, there are numerous continuous mappings from arbitrary convex quadrilaterals to squares that preserve the ratio of areas; the construction of Gromov after Knothe provides such a mapping with other nice properties \cite{Gr}. The idea of the mapping is to send slices of the first quadrilateral to slices of the second in a continuous way; see \cite{BrMo} for details on the specific construction. Therefore, given a probability density $f$ on the square, we obtain a unique probability density $f'$ on any convex quadrilateral, and if $f$ is uniform then so is $f'$.

We now formulate our decomposition process in precise terms. We begin with the unit square, a probability density $f$ on $(0,1)$, and a probability density $g$ on $(0,1)^2$. We independently select a point on each side according to $f$. Call these $A,B,C,D$. We then choose a point $E$ in the interior of the quadrilateral $ABCD$ according to the composition of $g$ with a mapping that preserves ratios of areas. We now connect $E$ to each of $A,B,C,D$ in order to decompose the square into four convex quadrilaterals. We then perform the same decomposition independently on each of these quadrilaterals, repeating this process until we obtain $4^N$ quadrilaterals.

Again, we consider the ratios of the areas of the sub-quadrilaterals to their parent, and use the resulting distribution to reframe this two-dimensional process as a decomposition of the rod. Because of the existence of continuous mappings that preserve ratios of areas, it is enough to consider the decomposition of the unit square. Each sub-quadrilateral can be thought of as having two triangular components: one outside $ABCD$ and one inside it. We will determine the distribution of the proportions of the sub-quadrilaterals by conditioning on the area of the inner quadrilateral $ABCD$. In terms of the decomposition of the rod, this corresponds to the following case of Lemma \ref{lem:rodjoining}. We split the rod into two pieces, one corresponding to the inner quadrilateral $ABCD$ and one corresponding to the outer area. We then split the former piece according to the areas of the inner triangles, and the outer piece according to the areas of the outer triangles, and then recombine the corresponding sub-pieces.

If $(a,b,c,d)$ are the proportions of the points on each side, then the outer triangles have areas $\frac{1}{2}a(1-b)$, $\frac{1}{2}b(1-c)$, $\frac{1}{2}c(1-d)$, $\frac{1}{2}d(1-a)$ and the quadrilateral $ABCD$ has area $\mathcal{A}=1- \frac{1}{2} (a(1-b)+b(1-c)+c(1-d)+d(1-a))$. Let $\pi$ denote the density of the probability distribution of $\mathcal{A}$. Let $h_1(\mathcal{A})$ denote the density of the distribution of the areas of the outer triangles conditional on $\mathcal{A}$. Note that if $f$ is continuous, then so are $\pi$ and $h_1$. Let $h_2$ denote the density of the distribution of the ratios of the areas of the inner triangles to $\mathcal{A}$. Pulling back using any continuous area preserving mapping, we may determine the distribution of $h_2$ by considering the case when $ABCD$ is the unit square. In this case, by selecting the point $(x,y)$ we produce triangles of area $\frac{1}{2}xy, \frac{1}{2}x(1-y), \frac{1}{2}(1-x)y, \frac{1}{2}(1-x)(1-y)$. The areas of the inner triangles in $ABCD$ are obtained by multiplying through by $\mathcal{A}$. Therefore, if $f$ and $g$ are both continuous, we see that $h_2$ is continuous.

Theorem \ref{thm:quadrilateral} follows from applying Lemma \ref{lem:rodjoining} with this $\pi$ above corresponding to the $\pi$ of the lemma and $h_1$ and $h_2$ corresponding to $G$ and $f$, respectively.
\end{proof}


\section{Proof of Theorem \ref{thm:detpiecesbenford}: Determinant Expansion} \label{sec:determinant}

The techniques introduced to prove that the continuous stick decomposition processes result in Benford behavior can be applied to a variety of dependent systems. To show this, we prove that the $n!$ terms of a matrix determinant expansion follow Benford's Law in the limit as long as the matrix entries are independent, identically distributed nice random variables. As the proof is similar to our previous results, we content ourselves with sketching the arguments, highlighting where the differences occur. See \cite{MaMil} for additional examples of Benford behavior in matrix ensembles.

Consider an $n \times n$ matrix $A$ with independent identically distributed entries $a_{pq}$ drawn from a continuous real valued density $f(x)$. Without loss of generality, we may assume that all entries of $A$ are non-negative. For $1 \leq i \leq n!$, let $X_{i,n}$ be the $i$\textsuperscript{\rm th} term in the determinant expansion of $A$. Thus $X_{i,n}=\prod_{p=1}^n a_{p\sigma_i(p)}$ where the $\sigma_i$'s are the $n!$ permutation functions on $\{1, \dots, n\}$.

We prove that the distribution of the significands of the sequence $\{X_{i,n}\}_{p = 1}^{n!}$ converges in distribution to Benford's Law when the entries of $A$ are drawn from a distribution $f$ that satisfies \eqref{eq:summeltransformsnozero} (with $h=f$). Recall that it suffices to show
\begin{enumerate}
\item $\lim\limits_{N \to \infty} \E[P_n(s)] \ = \ \log_{10}(s)$, and
\item $ \lim\limits_{N \to \infty} \var{P_n(s)} \ = \ 0$.
\end{enumerate}

We first quantify the degree of dependencies, and then sketch the proofs of the mean and variance. Fix $i \in \{1, \ldots, n!\}$ and consider the number of terms $X_{j,n}$ that share exactly $k$ entries of $A$ with $X_{i,n}$. Equivalently: If we permute the numbers $1, 2, \dots, n$, how likely is it that exactly $k$ are returned to their original starting position? This is the well known probl\'eme des rencontres (the special case of $k=0$ is counting derangements), and in the limit the probability distribution converges to that of a Poisson random variable with parameter 1 (and thus the mean and variance tend to 1; see \cite{HSW}). Therefore, if $K_{i,j}$ denotes the number of terms $X_{i,n}$ and $X_{j,n}$ share, the probability that $K_{i,j} > \log N$ is $o(1)$.

The determination of the mean follows as before. By linearity of expectation we have
\bea
\E[P_n(s)] &\ = \ &  \frac{1}{n!}\sum\limits_{i=1}^{n!} \E[\varphi_s(X_{i,n})].
\eea
It suffices to show that $\lim\limits_{n\to\infty}\E[\varphi_s(X_{i,n})]\ = \ \log_{10}s$. We have
\bea
\E[\varphi_s(X_{i,n})] &\ = \ &  \int_{a_{1\sigma_i(1)}}\int_{a_{2\sigma_i(2)}}\dotsm\int_{a_{n\sigma_i(n)}} \varphi_s\left(\prod_{p=1}^n a_{p\sigma_i(p)}\right)\nonumber\\
&&\cdot\prod_{p=1}^n f(a_{p\sigma_i(p)}) \ da_{1\sigma_i(1)} da_{2\sigma_i(2)}\dotsm da_{n\sigma_i(n)},
\eea
where $a_{p\sigma_i(p)}$ are the entries of $A$. As these random variables are independent and $f(x)$ satisfies \eqref{eq:summeltransformsnozero}, the convergence to Benford follows from \cite{JKKKM} (or see Appendix A of \cite{B--}). 

To complete the proof of convergence to Benford, we need to control the variance of $P_n(s)$. Arguing as before gives
\bea
\var{P_n(s)} &\ = \ &  \frac{\log_{10}(s)}{n!} + \frac{1}{(n!)^2} \left(\sum\limits_{\substack{i, j =1 \\ i \neq j}}^{n!}\E [\varphi_s(X_{i,n}) \varphi_s(X_{j,n})] \right) - \log_{10}^2 s+o(1).
\eea
We then mimic the proof from \S\ref{sec:varianceunrestricted}. There we used that, for a fixed $i$, the number of the $2^N$ pairs $(i,j)$ with $n$ factors not in common was $2^{n-1}$; in our case we use $K_{i,j}$ is approximately Poisson distributed to show that, with probability $1+o(1)$ there are at least $\log N$ different factors. The rest of the proof proceeds similarly.


\section{Future Work} \label{sec:conjectures}

Many of our results concern continuous decomposition models in which a stick is broken at a proportion $p$. We propose several variations of a discrete decomposition model in which a stick breaks into pieces of integer length, which we hope to return to in a future paper.

Consider the following additive decomposition process. Begin with a stick of length $L$ and uniformly at random choose an integer length $c \in [1, L-1]$ at which to cut the stick. This results in two pieces, one of length $c$ and one of length $L - c$. Continue this cutting process on both sticks, and only stop decomposing a stick if its length corresponds to a term in a given sequence $\{a_n\}$. As one cannot decompose a stick of length one into two integer pieces, sticks of length one stop decomposing as well.

\begin{conj}\label{conj:Benford}
The stick lengths that result from this decomposition process follow Benford's Law for many choices of $\{a_n\}$. In particular, if either (i) $\{a_n\} = \{2n\}$ or (ii) $\{a_n\}$ is the set of all prime numbers then the resulting stick lengths are Benford.
\end{conj}

More generally one can investigate processes where the stick lengths continue decomposing if they are in certain residue classes to a fixed modulus, or instead of stopping at the primes one could stop at a sequence with a similar density, such as $\lfloor n \log n \rfloor$. It is an interesting question to see, for a regularly spaced sequence, the relationship between the density of the sequence and the fit to Benford's law. While one must of course be very careful with any numerical exploration, as many processes are almost Benford (see for example the results on exponent random variables in \cite{LSE,MN2}), Theorem \ref{restricted thm} is essentially a continuous analogue of this problem when $a_n = 2n$, and thus provides compelling support.

Figure \ref{fig:stopOnEvens}(a) shows a histogram of chi-square values from numerical simulations with a $\chi^2_8$ distribution overlaid (as we have 9 first digits, there are 8 degrees of freedom). A stick of length approximately $10^4$ was decomposed according to the above cutting process, where $\{a_n\} = 2n$. A chi-square value\footnote{If $X_i$ are the simulation frequencies of each digit $i = 1, 2, \ldots, 9$ and $Y_i$ are the frequencies of each digit as predicted by Benford's Law, where $i = 1, \ldots, 9$, then the chi-squared value we calculated is given by $\sum_{i=1}^9 (X_i N - Y_i N)^2 / N Y_i$, where $N$ is the number of pieces.} was calculated by comparing the frequencies of each digit obtained from the simulation to the frequencies predicted by Benford's Law. If the simulated data is a good fit to Benford's Law, the $\chi^2$ values should follow a $\chi^2$ distribution. Examining the plot below shows that our numerical simulations support Conjecture \ref{conj:Benford}(i). Similarly, Figure \ref{fig:stopOnEvens}(b), which shows the chi-square values obtained when $\{a_n\}$ is the set of all primes,  supports Conjecture \ref{conj:Benford}(ii).

\begin{figure}[h]
\begin{center}
\scalebox{.3}{\includegraphics{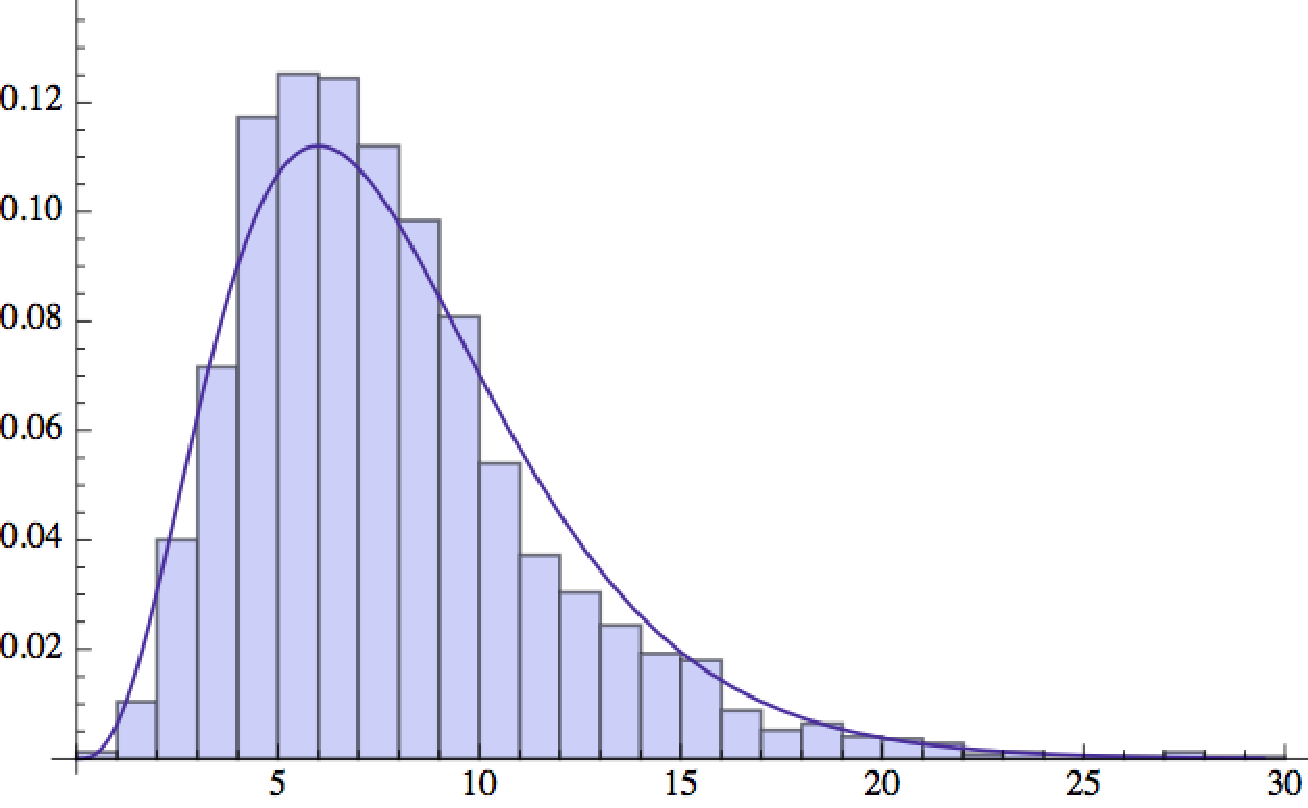}}\ \ \ \scalebox{.3}{\includegraphics{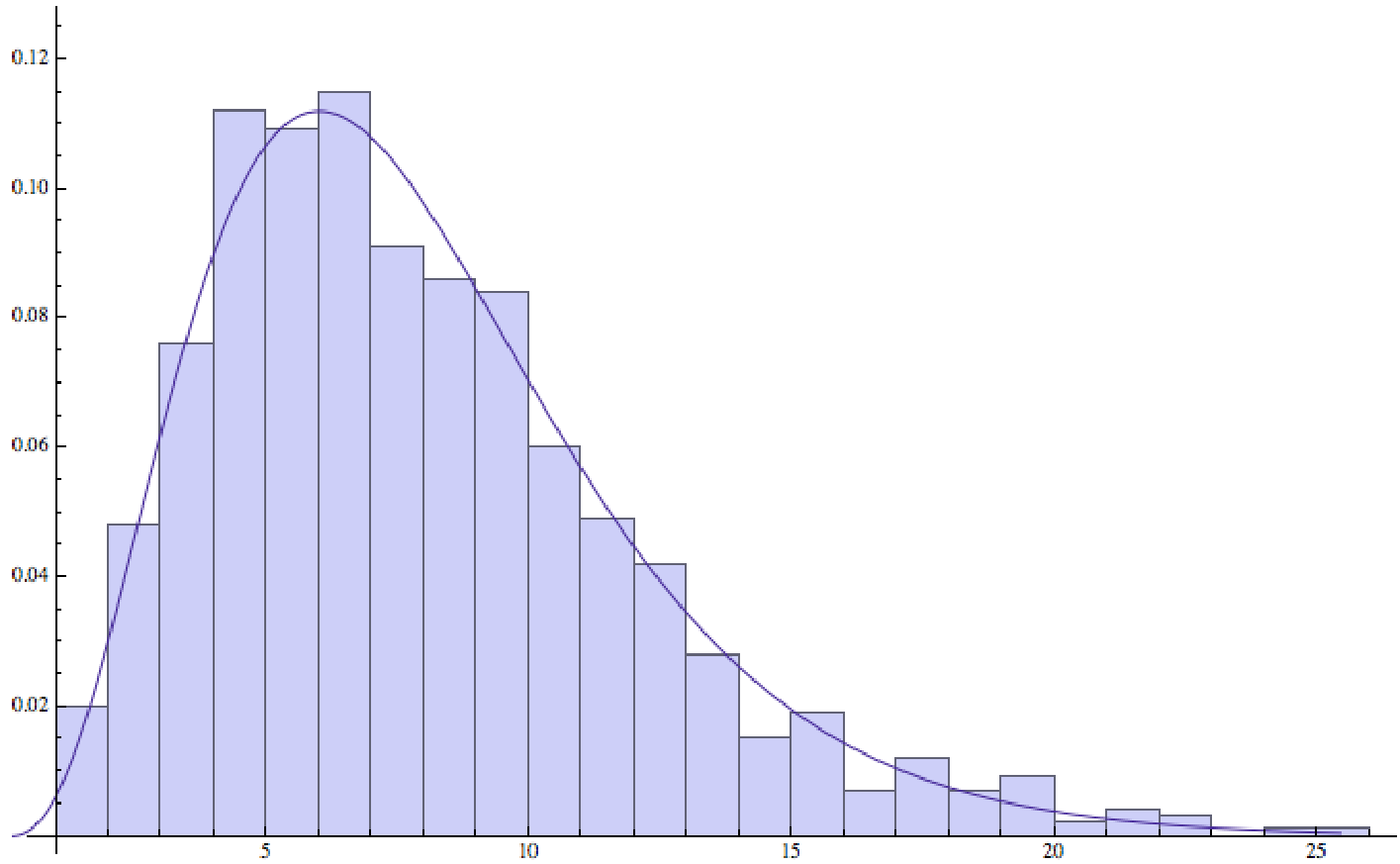}}
\caption{\label{fig:stopOnEvens} (a) $\{a_n = 2n\}$: 2500 chi-square values of stick lengths to Benford behavior after decomposition of sticks of length approximately $10^{7}$. (b) $\{a_n = p_n\}$ (the set of primes): 1000 chi-square values of stick lengths to Benford behavior after decomposition of sticks of length approximately $10^{500}$.}
\end{center}\end{figure}

\begin{conj} \label{conj:notBenford} Fix a monotonically increasing sequence $\{a_n\}$. Consider a decomposition process where a stick whose length is in the sequence does not decompose further. The process is not Benford if $\{a_n\}$ is any of the following: $\{n^2\}$, $\{2^n\}$, $\{F_n\}$ where $F_n$ is the $n$\textsuperscript{{\rm th}} Fibonacci number. \end{conj}

More generally, we do not expect Benford behavior if the sequence is too sparse, which is the case for polynomial growth in $n$ (when the exponent exceeds 1) or, even worse, geometric growth.

Figure \ref{fig:stopOnSquares} features plots of the observed digit frequencies vs the Benford probabilities. These plots also give the total number of fragmented pieces, the number of pieces whose lengths belong to the stopping sequence, and the chi-squared value. Notice that the stopping sequences in Conjecture \ref{conj:notBenford} result in stick lengths with far too high a probability of being of length $1$. This fact leads us to believe that the sequences are not ``dense'' enough to result in Benford behavior.

In addition to proofs of the conjectures, future work could include further exploration of the relationship between stopping sequence density and Benford behavior.

\begin{figure}[h]
\begin{center}
\scalebox{.3}{\includegraphics{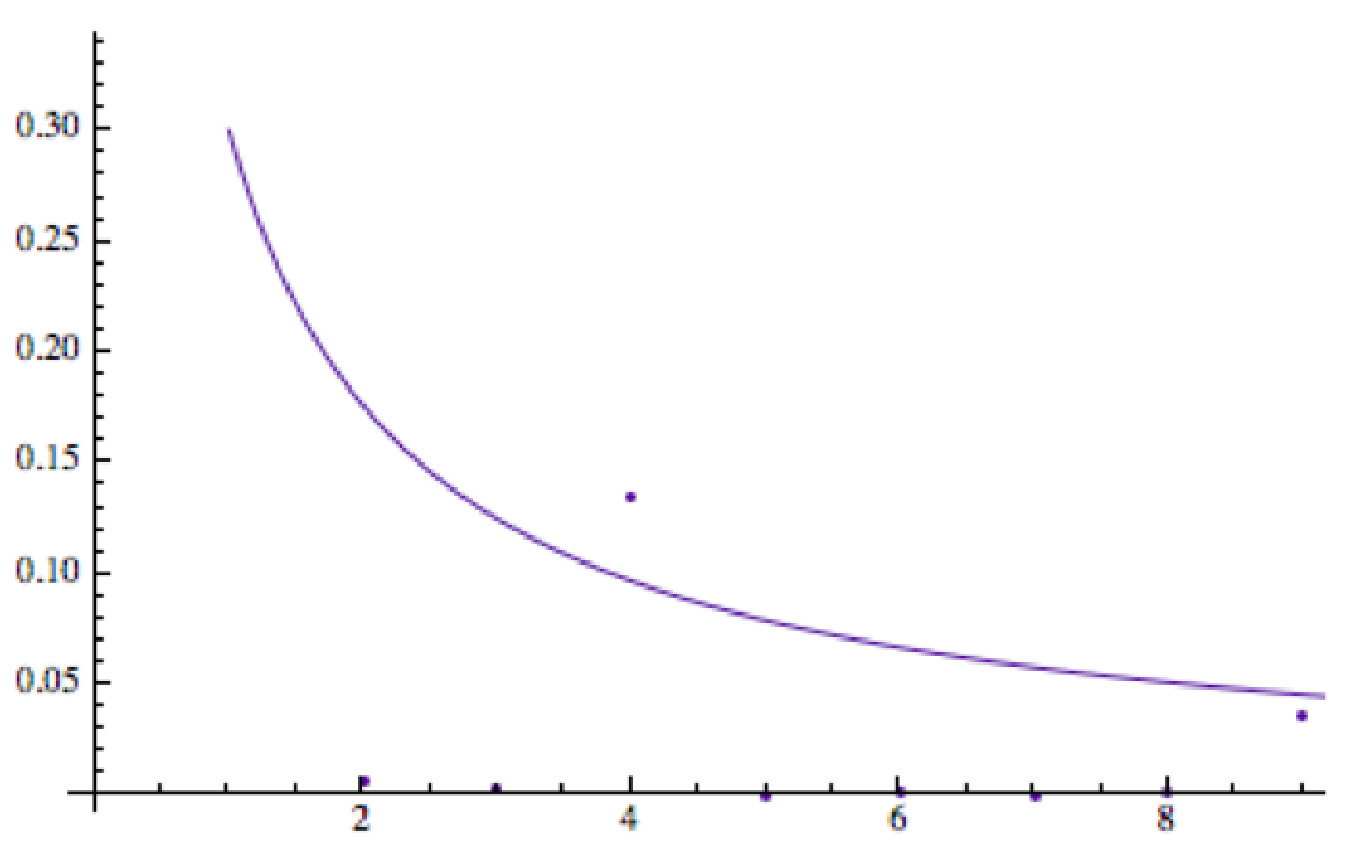}} \ \ \ \scalebox{.3}{\includegraphics{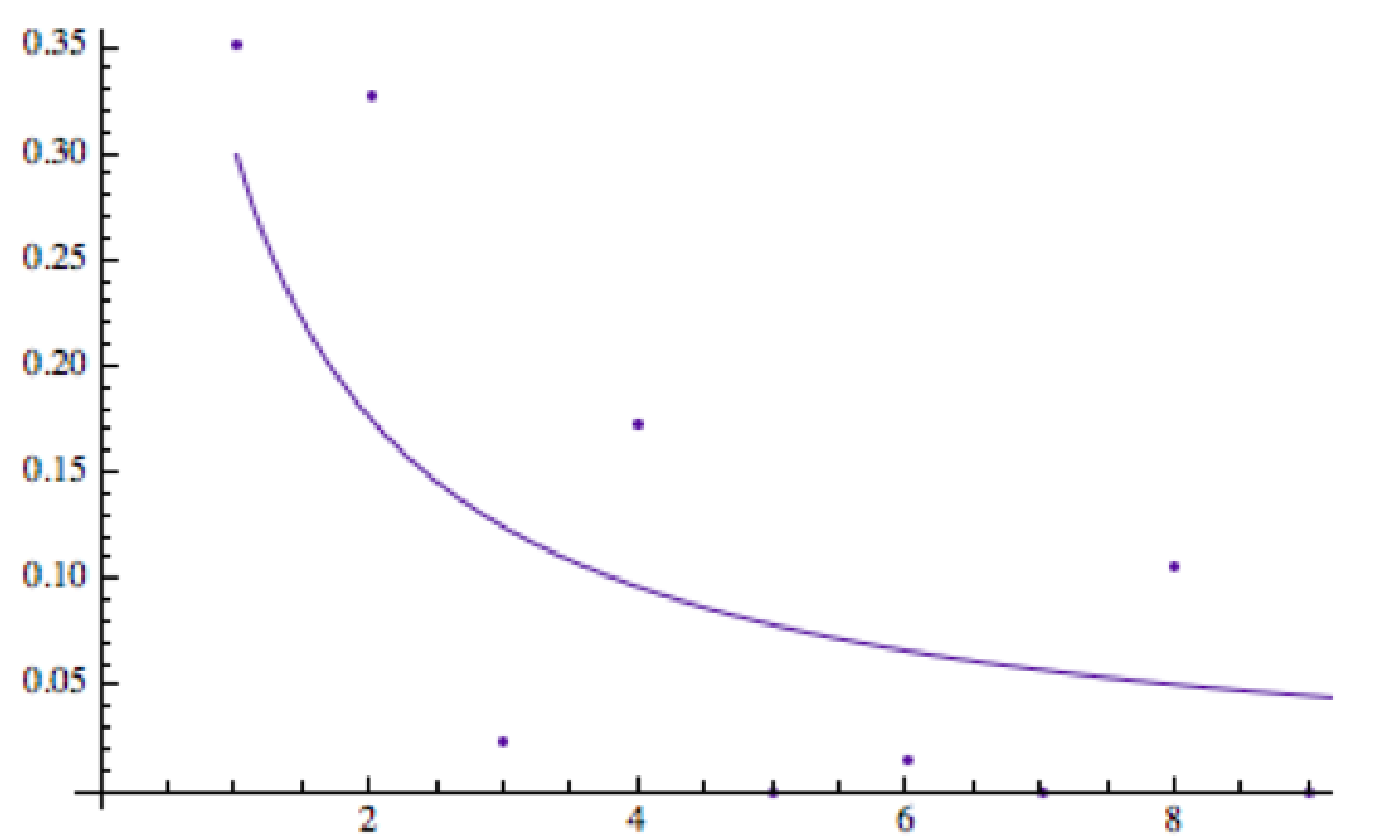}} \ \ \ \scalebox{.3}{\includegraphics{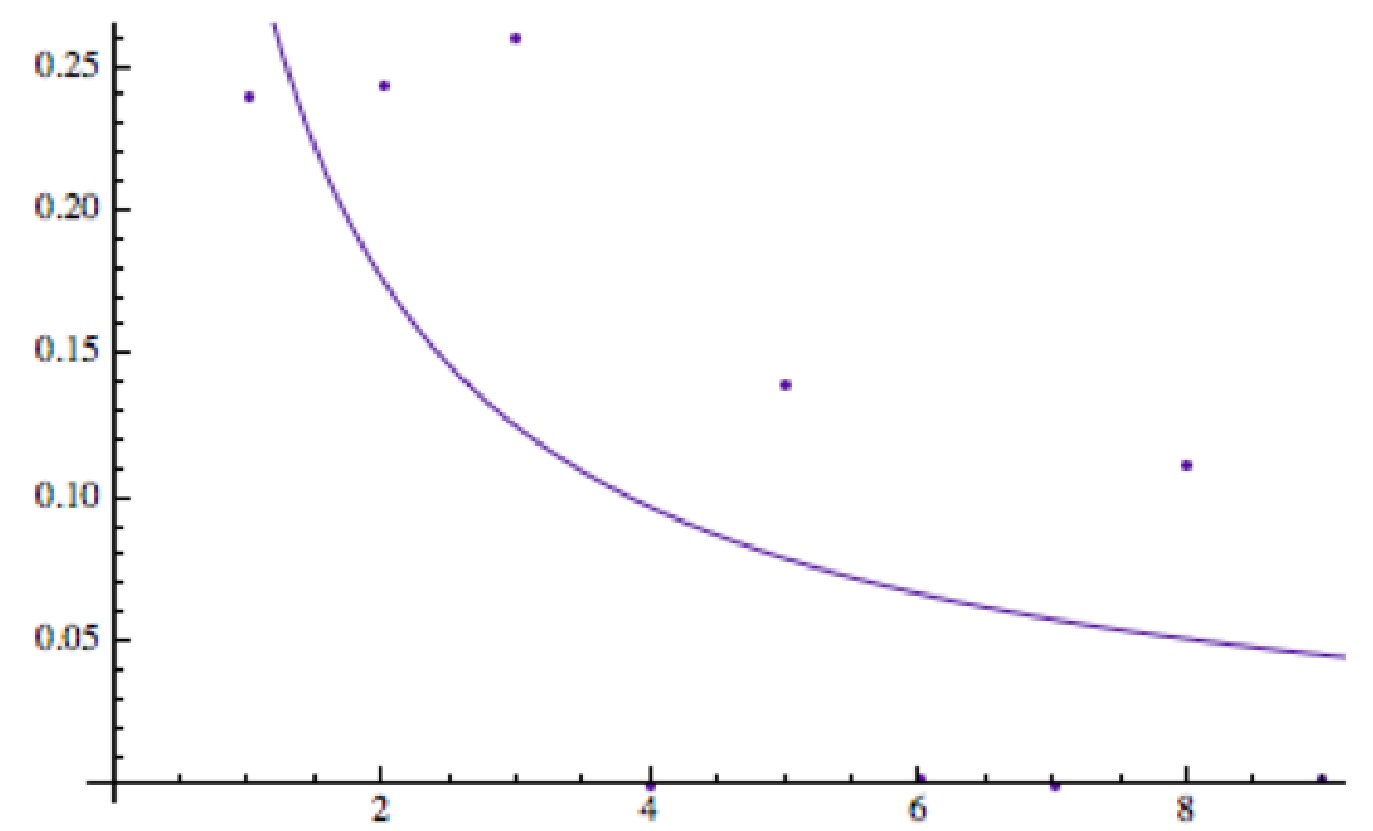}}
\caption{\label{fig:stopOnSquares} (a) $\{a_n = n^2\}$: Distribution of stick lengths after decomposing sticks of starting length approximately $10^{6}$. (b) $\{a_n = 2^n\}$: Distribution of stick lengths after decomposing 1000 sticks of starting length approximately $10^{500}$. (c) $\{a_n = F_n\}$, the Fibonacci Numbers: Distribution of stick lengths after decomposing 500 sticks of starting length approximately $10^{500}$.}
\end{center}\end{figure}


\appendix


\section{Non-Benford Behavior}\label{benfordnotbenford}

We consider the following generalization of the Unrestricted Decomposition Model of Theorem \ref{thm:unrestricted}, where now the proportions at level $n$ are drawn from a random variable with density $\phi_n$. If there are only finitely many densities and if the Mellin transform condition is satisfied, then the sequence of stick lengths converges to Benford; if there are infinitely many possibilities, however, then it is possible to obtain non-Benford behavior.

Specifically, we give an example of a sequence of distributions $\mathcal{D}$ with the following property: If all cut proportions are drawn from any one distribution in $\mathcal{D}$ then the resulting stick lengths converge to Benford's Law, but if all cut proportions in the $n$\textsuperscript{\rm th} level are chosen from the $n$\textsuperscript{\rm th} distribution in $\mathcal{D}$, the stick lengths do not exhibit Benford behavior.

It is technically easier to work with the densities of the logarithms of the cut proportions. Fix $\delta>0$  and choose a sequence of $\epsilon_n$'s so that they monotonically decrease to zero and \be \epsilon_n\ <\ \min \left(\sqrt{\frac{3}{20\pi^2(n+1)^2}}, \frac12 \frac{\delta}{2^n}, \frac{\log \frac12}2, \frac{1-\log \frac12}{2}\right);\ee these values and the meaning of these constants will be made clear during the construction. Consider the sequence of distributions with the following densities
\bea
\twocase{\phi_n(x) \ = \ }{\frac{1}{2 \epsilon_n}}{if $\left|x - \log\frac{1}{2}\right| < \epsilon_n$}{0}{otherwise;}
\eea these will be the densities of the logarithms of the cut proportions.

If the logarithm of the cut proportion is drawn solely from the distribution $\phi_k$ for a fixed $k$, then as the number of iterations of
this cut process tends to infinity the resulting stick lengths converge to Benford's Law. This follows immediately from Theorem \ref{thm:unrestricted} as the associated densities of the cut proportions satisfy the Mellin condition \eqref{eq:summeltransformsnozero}.

We show that if the logarithms of the cut proportions for the $n$\textsuperscript{th} iteration of the decomposition process is drawn from
$\phi_n$ then the resulting stick lengths do not converge to Benford behavior for certain $\delta$. What we do is first show that the distribution of one stick piece (say the resulting piece from always choosing the left cut at each stage) is non-Benford, and then we prove that the ratio of the lengths of any two final pieces is approximately 1. The latter claim is reasonable as our cut proportions are becoming tightly concentrated around 1/2; if they were all exactly 1/2 then all final pieces would be the same length.

Instead of studying the Mellin transforms of the cut proportions we can study the Fourier coefficients of the logarithms of the proportions (again, this is not surprising as the two transforms are related by a logarithmic change of variables). For a function $\phi$ on $[0,1]$, its $n$\textsuperscript{th} Fourier coefficient is \be \widehat{\phi}(n) \ := \ \int_0^1 \phi(x) e^{-2\pi i n x} dx.\ee Miller and Nigrini \cite{MN1} proved that if $R_1, \dots, R_M$ are continuous independent random variables with $g_m$ the density of $\log_{10} S_{10}(|R_m|)$, then the product $R_1 \cdots R_M$ converges to Benford's law if and only if for each non-zero $n$ we have $\lim_{M\to 0} \widehat{g_1}(n) \cdots \widehat{g_M}(0) = 0$.

We show that if we take our densities to be the $\phi_n$'s that the limit of the product of the Fourier coefficients at 1 does not tend to 0. We have
\bea
\widehat{\phi_{n}}(1) \ = \   \frac{1}{2 \epsilon_n}
\int_{\frac{1}{2} - \epsilon_{n}}^{\frac{1}{2}
+ \epsilon_{n}} e^{-2 \pi i x} dx \ = \ -1+\frac{2\pi^2\epsilon_{n}^2}{3} +O(\epsilon^4_{n}).
\eea
As we are assuming $\epsilon_n < \sqrt{\frac{3}{20\pi^2(n+1)^2}}$, we find for $n$ large that \be \left|\widehat{\phi_n}(1)\right| \ \ge \ \frac{n^2 + 2n}{(n+1)^2}. \ee As \be \lim_{N \to \infty} \prod_{n=1}^N \frac{n^2+2n}{(n+1)^2} \  = \ \frac12, \ee we see that \be \lim_{N \to\infty} \prod_{n=1}^N \left|\widehat{\phi_1}(1) \cdots \widehat{\phi_N}(1)\right| \ > \ 0. \ee

This argument shows that the product of $N$ random variables, where the logarithm of the $n$\textsuperscript{th} variable is drawn from the distribution $\phi_{n}$, is not Benford. This product is analogous to the length of one of the $2^N$ sticks that are created after $N$ iterations of our cut process. To show the entire collection of stick lengths does not follow a Benford distribution, we argue that all of the lengths are virtually identical because the cutting proportion tends to $1/2$ (specifically, that the ratio of any two lengths is approximately 1).

Let us denote the lengths of the sticks left after the $N$\textsuperscript{th} iteration of our cutting procedure by $X_{N,i}$, where $i = 1, \ldots, 2^N$. Each length is a product of $N$ random variables; the $n$\textsuperscript{th} term in the product is either $p_n$ (or $1-p_n$), where the logarithm of $p_n$ is drawn from the distribution $\phi_n$. Proving the ratio of any two lengths is approximately 1 is the same as showing that $\log (X_{N_i}/X_{N,j})$ is approximately 0. If we can show the largest ratio has a logarithm close to 0 then we are done. The largest (and similarly smallest) ratio comes when $X_{N,i}$ and $X_{N,j}$ have no terms in common, as then we can choose one to have the largest possible cut at each stage and the other the smallest possible cut. Thus $X_{N,i}$ always has the largest possible cut; as the largest logarithm at the $n$\textsuperscript{th} stage is $\epsilon_n + \log \frac12$, its proportion at the $n$\textsuperscript{th} level is $e^{\epsilon_n + \log\frac12}$. Similarly $X_{N,j}$ always has the smallest cut, which at the $n$\textsuperscript{th} level is $e^{-\epsilon_n + \log \frac12}$. Therefore \be \log\frac{X_{N,i}}{X_{N,j}} \ = \ \log \prod_{n=1}^N \frac{e^{\epsilon_n + \log\frac12}}{e^{-\epsilon_n + \log\frac12}} \ = \ \sum_{n=1}^N 2\epsilon_n; \ee as we chose $\epsilon_n < \frac12 \frac{\delta}{2^n}$, the maximum ratio between two pieces is at most $\delta$. By choosing $\delta$ sufficiently small we can ensure that all the pieces have approximately the same significands (for example, if $\delta < 10^{-2014}$ then we cannot get all possible first digits).

\section{Notes on Lemons' Work}\label{sec:noteslemons}

In his 1986 paper, \emph{On the Number of Things and the Distribution of First Digits}, Lemons \cite{Lem} models a particle decomposition process and offers it as evidence for the prevalence of Benford behavior, arguing that many sets that exhibit Benford behavior are merely the breaking down of some conserved quantity.  However, Lemons is not completely mathematically rigorous in his analysis of the model (which he states in the paper), and glosses over several important technical points. We briefly mention some issues, such as concerns about the constituent pieces in the model as well as how the initial piece decomposes. We discuss our resolutions of these issues as well as their impacts on the behavior of the system.

The first issue in Lemons' model concerns the constituent pieces. He assumes the set of possible piece sizes is bounded above and below and is drawn from a finite set, eventually specializing to the case where the sizes are in a simple arithmetic progression (corresponding to a uniform spacing), and then taking a limit to assume the pieces are drawn from a continuous range. In this paper, we allow our piece lengths to be drawn continuously from intervals at the outset, and not just in the limit. This removes some, but by no means all, of the technical complications. One must always be careful in replacing discrete systems with continuous ones, especially as there can be number-theoretic restrictions on which discrete systems have a solution. Modeling any conserved quantity is already quite hard with the restriction that the sum of all parts must total to the original starting amount; if the pieces are forced to be integers then certain number theoretic issues arise. For example, imagine our pieces are of length 2, 4 or 6, so we are trying to solve $2x_1 + 4x_2 + 6x_3 = n$. There are no solutions if $n = 2017$, but there are 85,345 if $n$ is 2018. By considering a continuous system from the start, we avoid these Diophantine complications. We hope to return to the corresponding discrete model in a sequel paper.


A second issue missing from Lemons argument is how the conserved quantity, the number of pieces, and the piece sizes should be related.  The continuous limit thus requires consideration of three quantities.  Without further specification of their relative scaling, the power law distribution (which leads to Benford behavior) is but one possible outcome.  Statistical models of the fragmentation of a conserved quantity based on integer partitioning have been constructed \cite{ChaMek,LeeMek,Mek}.  These models can lead to a power law distribution but only for special weightings for the different partitions.  Whether this distribution can be obtained from equally weighted partitions (as used in Lemons argument) is an important question, to which we hope to return.

A related issue is that it is unclear how the initial piece breaks up. The process is not described explicitly, and it is unclear how likely some pieces are relative to others. Finally, while he advances heuristics to determine the means of various quantities, there is no analysis of variances or correlations. This means that, though it may seem unlikely, the averaged behavior could be close to Benford while most partitions would be far from Benford.\footnote{Imagine we toss a coin one million times, always getting either all heads or all tails. Let's say these two outcomes are equally likely. If we were to perform this process trillions of times, the total number of heads and tails would be close to each other; however, no individual experiment would be close to 50\%.}

These are important issues, and their resolution and model choice impacts the behavior of the system. We mention a result from Miller-Nigrini \cite{MN2}, where they prove that while order statistics are close to Benford's Law (base 10), they do not converge in general to Benford behavior. In particular, this means that if we choose $N$ points randomly on a stick of length $L$ and use those points to partition our rod, the distribution of the resulting piece sizes \emph{will not} be Benford. Motivated by this result and Lemons' paper, we instead considered a model where at stage $N$ we have $2^N$ sticks of varying lengths, and each stick is broken into two smaller sticks by making a cut on it at some proportion. Each cut proportion is chosen from the unit interval according to a density $f$. Dependencies clearly exist within this system as the lengths of final sticks must sum to the length of the starting stick.


\ \\

\end{document}